\documentclass[11pt,letterpaper]{article}

\usepackage[T1]{fontenc}
\usepackage{lmodern}
\usepackage[utf8]{inputenc} 
\usepackage[T1]{fontenc}    
\usepackage{verbatim}
\usepackage[english]{babel}
\usepackage{amsthm}
\usepackage{amsmath}
\usepackage{algcompatible}
\usepackage[noend]{algpseudocode}
\usepackage{algorithm}

\usepackage{url}            
\usepackage{booktabs}
\usepackage{enumerate}
\usepackage{graphicx}
\usepackage{amssymb}
\usepackage{xcolor}
\usepackage{multirow}
\usepackage{booktabs}
\usepackage{bbm}
\usepackage{caption}
\usepackage{subcaption}
\usepackage{lipsum}
\usepackage{color, colortbl}
\usepackage[english]{babel}

\usepackage[margin=1.0in]{geometry}

\newtheorem{theorem}{Theorem}
\newtheorem{corollary}{Corollary}

\newtheorem{definition}{Definition}
\newtheorem{problem}{Problem}
\newtheorem{assumption}{Assumption}

\newtheorem{remark}{Remark}

\title{Learning Lyapunov Functions for Hybrid Systems}

\author{Shaoru Chen, Mahyar Fazlyab, Manfred Morari, George J. Pappas, Victor M. Preciado
	\thanks{Shaoru Chen, Manfred Morari, George J. Pappas, and Victor M. Preciado are with the Department of Electrical and Systems Engineering, University of Pennsylvania. Email: \{srchen, morari, pappasg, preciado\}@seas.upenn.edu. Mahyar Fazlyab is with the Mathematical Institute for Data Science, Johns Hopkins University. Email: mahyarfazlyab@jhu.edu}} 

\date{}
\begin{document}
\pagestyle{plain}
\maketitle

\begin{abstract}
We propose a sampling-based approach to learn Lyapunov functions for a class of discrete-time autonomous hybrid systems that admit a mixed-integer representation. Such systems include autonomous piecewise affine systems, closed-loop dynamics of linear systems with model predictive controllers, piecewise affine/linear complementarity/mixed-logical dynamical system in feedback with a ReLU neural network controller, etc. The proposed method comprises an alternation between a learner and a verifier to find a valid Lyapunov function inside a convex set of Lyapunov function candidates. In each iteration, the learner uses a collection of state samples to select a Lyapunov function candidate through a convex program in the parameter space. The verifier then solves a mixed-integer quadratic program in the state space to either validate the proposed Lyapunov function candidate or reject it with a counterexample, i.e., a state where the Lyapunov condition fails. This counterexample is then added to the sample set of the learner to refine the set of Lyapunov function candidates. By designing the learner and the verifier according to the analytic center cutting-plane method from convex optimization, we show that when the set of Lyapunov functions is full-dimensional in the parameter space, our method finds a Lyapunov function in a finite number of steps. We demonstrate our stability analysis method on closed-loop MPC dynamical systems and a ReLU neural network controlled PWA system.


%
%
%

\end{abstract}

\section{Introduction}

Hybrid systems have become widespread within the systems and control community in the last decades thanks to their flexibility in modeling the interaction of continuous and discrete dynamical systems that frequently arise in cyber-physical systems (CPS) \cite{bemporad1999control,borrelli2017predictive}. As today's cyber-physical systems are getting more complex, developing new methods for design, analysis, and control of hybrid systems is increasingly important.

Analysis and control design for general hybrid systems is challenging and therefore, various methods have been proposed to tackle special classes of hybrid systems such as linear complementarity (LC) systems~\cite{heemels2000linear, van1998complementarity}, mixed logical dynamical (MLD) systems~\cite{bemporad1999control}, and piecewise affine (PWA) systems~\cite{sontag1981nonlinear}. While these classes are mathematically equivalent~\cite{heemels2001equivalence}, their representation could have a high impact on their numerical tractability. When it comes to stability analysis, various methods have been proposed for PWA systems~\cite{biswas2005survey, lin2009stability, sun2010stability} while methods that directly deal with MLD and LC systems are relatively scarce. Nevertheless, stability analysis tools for PWA systems are applicable to MLD and LC systems as they can be transformed into PWA systems~\cite{heemels2001equivalence}.

Transforming various types of hybrid systems into a PWA representation for stability analysis may not always be efficient. For example, for a PWA system in feedback with a ReLU neural network controller, although the closed-loop dynamics is PWA, identifying the PWA representation may be tedious and the stability analysis task may become very challenging since the number of partitions generated by the ReLU network can be very large~\cite{pascanu2013number}. 


In this paper, we propose a learning-based approach to stability analysis of hybrid systems that admit a mixed-integer formulation. These systems include PWA, MLD, LC systems and ReLU networks. Our method comprises a learner and a verifier, which iteratively search for a Lyapunov function from a target class $\mathcal{F}$ of Lyapunov functions (e.g., quadratic or piecewise quadratic). In each iteration, the learner uses a set of samples of the hybrid system to localize $\mathcal{F}$ by a convex set $\tilde{\mathcal{F}} \supseteq \mathcal{F}$ and then solves a semidefinite program (SDP) to select a Lyapunov function candidate from $\tilde{\mathcal{F}}$. The verifier then solves a mixed-integer program in the state space to either validate the Lyapunov function candidate or reject it with a counterexample, i.e., a state where the Lyapunov condition fails. This counterexample is then added to the sample set of the learner to refine the set $\tilde{\mathcal{F}}$ of Lyapunov function candidates. By designing the alternation between the learner and the verifier according to the analytic center cutting-plane method (ACCPM), we show that when the set $\mathcal{F}$ of Lyapunov functions is full-dimensional and contains a norm ball with radius $\epsilon > 0$ in the parameter space, our method is guaranteed to find a Lyapunov function in $\mathcal{O}({n}^3/\epsilon^2)$ steps, where $n$ is the ambient dimension.

\subsection{Related work}
\emph{LMI-based stability analysis of PWA systems}:
Among various stability analysis methods for PWA dynamical systems~\cite{lin2009stability, sun2010stability}, linear matrix inequality (LMI)-based approaches are relatively prominent. These methods construct an SDP whose solution gives a valid Lyapunov function. For continuous-time PWA systems, LMI-based approaches to synthesize piecewise affine \cite{johansson2003piecewise}, piecewise quadratic (PWQ) \cite{johansson1997computation} and piecewise polynomial Lyapunov functions~\cite{prajna2003analysis} have been proposed. The adaptation of these Lyapunov function synthesis methods to handle discrete-time PWA systems is summarized in~\cite{biswas2005survey}. For discrete-time PWA systems, a common feature of the LMI-based methods is computing the transition map between all pairs of modes. This step may become time-consuming when the number of modes is large. 

\medskip

\noindent \emph{Sampling-based Synthesis Methods}: The iterative approach of alternating between a learning module and a verification module to synthesize a certificate for control systems is known as the Counter-Example Guided Inductive Synthesis (CEGIS) framework proposed by~\cite{solar2006combinatorial, solar2008program} in the verification community. The application of CEGIS to Lyapunov function synthesis for continuous-time nonlinear autonomous systems can be found in~\cite{ahmed2020automated, abate2020formal,kapinski2014simulation} using Satisfiability Modulo Theory (SMT) solvers for verification. In general, the termination of the iterative procedures in these works is not guaranteed. Notably, Ravanbakhsh et al.~\cite{ravanbakhsh2019learning} apply the CEGIS framework to synthesize control Lyapunov functions for nonlinear continuous-time systems and provide finite-step termination guarantees for the iterative algorithm through careful design of the learner which essentially implements the maximum volume ellipsoid cutting-plane method~\cite{Tarasov1988}. Our work differs from~\cite{ravanbakhsh2019learning} in the algorithm design and the application on hybrid systems. Other than the CEGIS framework, a learning-based approach to synthesize control barrier functions for hybrid systems is proposed in~\cite{lindemann2020learning}.

\subsection{Notations}
We denote the set of real numbers by $\mathbb{R}$, the set of integers by $\mathbb{Z}$, the $n$-dimensional real vector space by $\mathbb{R}^n$, and the set of $n \times m$-dimensional real matrices by $\mathbb{R}^{n \times m}$. The standard inner product between two matrices $A, B \in \mathbb{R}^{n \times m}$ is given by $\langle A, B \rangle = \text{tr}(A^\top B)$ and the Frobenius norm of a matrix $A \in \mathbb{R}^{n \times m}$ is given by $\lVert A \rVert_F = (\text{tr}(A^\top A))^{1/2}$. Denote $\mathbb{S}^n$ the set of $n \times n$-dimensional symmetric matrices, and $\mathbb{S}^n_{+}$ ($\mathbb{S}^n_{++}$) the set of $n \times n$-dimensional positive semidefinite (definite) matrices. Given a set $\mathcal{S} \subseteq \mathbb{R}^{n_x + n_y}, \text{Proj}_x (\mathcal{S}) = \{ x \in \mathbb{R}^{n_x} \vert \exists y \in \mathbb{R}^{n_y} \text{ s.t. } (x, y) \in \mathcal{S}\}$ denotes the orthogonal projection of $\mathcal{S}$ onto the subspace $\mathbb{R}^{n_x}$. We denote $\text{int}(\mathcal{S})$ the set of all interior points in $\mathcal{S}$.



\section{Mixed-integer formulation of hybrid systems}
Consider a discrete-time autonomous hybrid system
\begin{equation}\label{eq:nonlinear_sys}
x_+ = f(x)
\end{equation}
where $x \in \mathbb{R}^{n_x}$ is the state and $f:\mathbb{R}^{n_x} \mapsto \mathbb{R}^{n_x}$ is a continuous function. Without loss of generality, assume system~\eqref{eq:nonlinear_sys} has an equilibrium at the origin, i.e., $0 = f(0)$.
Let $\mathcal{R} \subset \mathbb{R}^{n_x}$ be the domain of the system and $\mathcal{R}$ is a compact set which contains the origin in its interior. Denote $x_k$ the state of system~\eqref{eq:nonlinear_sys} at time $k$ and $x_0$ the initial state. The nonlinear dynamics $x_+ = f(x)$ with domain $\mathcal{R}$ can be equivalently described through its graph defined as
\begin{equation}\label{eq:graph}
\text{gr}(f) = \{(x, y) \in \mathbb{R}^{n_x} \times \mathbb{R}^{n_x} \vert x \in \mathcal{R}, y = f(x) \}.
\end{equation}

In this paper, we study the Lyapunov stability of the origin of~\eqref{eq:nonlinear_sys} for a class of hybrid systems that admit a mixed-integer formulation.

\begin{definition}[Mixed-integer formulation of a set~\cite{marcucci2019mixed}] \label{def:MIF}
	For a set $\mathcal{Q} \subset \mathbb{R}^{n_z}$, consider the set $\mathcal{L}_\mathcal{Q} \subseteq \mathbb{R}^{n_z} \times \mathbb{R}^{ n_\lambda} \times \mathbb{Z}^{n_\mu}$ in a lifted space given by 
	\begin{equation}
	\mathcal{L}_\mathcal{Q} = \{ (z \in \mathbb{R}^{n_z}, \lambda \in \mathbb{R}^{n_\lambda}, \mu \in \mathbb{Z}^{n_\mu})  \vert \ell(z, \lambda, \mu) \leq v\},
	\end{equation}
	with a function $\ell: \mathbb{R}^{n_z} \times \mathbb{R}^{ n_\lambda} \times \mathbb{Z}^{n_\mu} \rightarrow \mathbb{R}^{n_\ell}$ and a vector $v \in \mathbb{R}^{n_\ell}$. The set $\mathcal{L}_\mathcal{Q}$ is a mixed-integer formulation of $\mathcal{Q}$ if $\text{Proj}_z(\mathcal{L}_\mathcal{Q}) = \mathcal{Q}$. If the function $\ell$ is linear, we call the related formulation mixed-integer linear (MIL). 
\end{definition} 
In the next subsections, we show how to find mixed-integer formulations for PWA, MLD, LC systems and ReLU networks.

\subsection{Piecewise affine systems}
\label{sec:PWA_representation}
Consider a discrete-time piecewise affine system with control inputs
\begin{equation} \label{eq:pwa_dynamics}
x_+ = \psi_i(x, u) = A_i x + B_i u + c_i, \forall (x, u) \in \mathcal{R}_i,
\end{equation}
where $\mathcal{R}_i = \{(x, u) \in \mathbb{R}^{n_x} \times \mathbb{R}^{n_u} \vert F_i x + G_i u \leq h_i\}$ for $i \in \mathcal{I} = \{1, 2, \allowdisplaybreaks \cdots, N_{mode}\}$ are polyhedral partitions of the state-input space $\mathcal{R} = \bigcup_{i \in \mathcal{I}} \mathcal{R}_i$. We assume that the partitions $\mathcal{R}_i$ are bounded for all $i \in \mathcal{I}$. To make the PWA system~\eqref{eq:pwa_dynamics} well-posed, we assume that $\text{int}(\mathcal{R}_i) \cap \text{int}(\mathcal{R}_j) = \emptyset, \forall i \neq j$ and $f_i(x,u) = f_j(x,u), \forall (x,u) \in \mathcal{R}_i \cap \mathcal{R}_j$ if the intersection is not an empty set. 

We denote the PWA dynamics~\eqref{eq:pwa_dynamics} collectively as $x_+ = \psi(x, u)$. The graph of $\psi(x,u)$ is given by
$\text{gr}(\psi) = \bigcup_{i\in \mathcal{I}} \text{gr} (\psi_i)$, 
 where each graph $\text{gr}(\psi_i)$ is defined as
\begin{equation}
\text{gr}(\psi_i) = \{(x, u, x_+) \vert Q_i [x^\top \ u^\top \ x_+^\top]^\top \leq q_i \},
\end{equation}
with
\begin{equation} 
Q_i = \begin{bmatrix}
A_i & B_i & -I \\ 
-A_i & -B_i & I \\
F_i & G_i & 0 
\end{bmatrix}, \quad 
q_i = \begin{bmatrix}
-c_i \\ c_i \\h_i 
\end{bmatrix}.
\end{equation}

In this paper, we apply a disjunctive programming-based formulation~\cite{marcucci2019mixed} which states that $x_+ = \psi(x,u)$ is equivalent to the following set of constraints~\cite{marcucci2019mixed}
\begin{equation} \label{eq:pwa_MIL_constr}
\begin{aligned}
& F_i x_i + G_i u_i \leq \mu_i h_i, \ \mu_i \in \{0, 1\}, \ \forall i \in \mathcal{I}, \\
& 1 = \sum_{i \in \mathcal{I}} \mu_i, \ x = \sum_{i \in \mathcal{I}} x_i, \ u =\sum_{i\in \mathcal{I}} u_i \\
& x_+ = \sum_{i\in \mathcal{I}} (A_i x_i + B_i u_i + \mu_i c_i).
\end{aligned}
\end{equation}
In the disjunctive programming formulation~\eqref{eq:pwa_MIL_constr}, we have $N_{mode}$ binary variables $\mu_i$ and $2N_{mode}$ auxiliary continuous variables $\{x_i, u_i\}$. The binary variable $\mu_i$ can be interpreted as the indicator of the mode where the state-input pair $(x, u)$ lives in. Note that $\mu_i = 1$ imposes $\mu_j = 0, \forall j \neq i$. By the boundedness of the partitions $\mathcal{R}_j$, we have $F_j x_j + G_j u_j \leq \mu_j h_j = 0 \Rightarrow x_j = 0, u_j = 0$, and correspondingly $x = \sum_{k \in \mathcal{I}} x_k = x_i, u = \sum_{k \in \mathcal{I}} u_k = u_i, x_+ = \sum_{k \in \mathcal{I}} (A_k x_k + B_k u_k +  \mu_k c_k) = A_i x_i + B_i u_i + c_i$. We can also obtain a mixed-integer formulation of the PWA dynamics~\eqref{eq:pwa_MIL_constr} through the big-M method~\cite{vielma2015mixed}.

When the PWA system~\eqref{eq:pwa_dynamics} is interconnected with a controller which also has a mixed-integer formulation, we can describe the closed-loop dynamics through mixed-integer constraints. When an autonomous PWA system is considered, we obtain its mixed-integer formulation by removing the control input related variables in~\eqref{eq:pwa_MIL_constr}.

\subsection{Linear complementarity systems}
\label{sec:LCS_representation}
Consider a discrete-time linear complementarity system~\cite{heemels2001equivalence, heemels2000linear} 
\begin{subequations} \label{eq:LCS}
	\begin{align}
	\begin{split}
	&x_+ = Ax + B_1 u + B_2 w
	\end{split}\\
	\begin{split}
	&v = E_1 x + E_2 u + E_3 w + g_4
	\end{split}\\
	\begin{split} \label{eq:LC_constr}
	&0 \leq v \perp w \geq 0
	\end{split}
	\end{align}
\end{subequations}
where $v, w \in \mathbb{R}^s$ and $\perp$ denotes that $v^\top w = 0$. The complementarity constraint~\eqref{eq:LC_constr} can be equivalently formulated as a set of mixed-integer linear constraints through the big-M method~\cite{vielma2015mixed} as
\begin{equation} \label{eq:MIL_LCS}
\begin{aligned}
v_i w_i = 0 \Leftrightarrow 0 \leq v_i \leq \mu_i M_{1,i}, \ 0 \leq w_i \leq (1 - \mu_i) M_{2,i},
\end{aligned}
\end{equation}
where $\mu_i \in \{0, 1\}$ and the subscript $i$ denotes the $i$-th entry of the variable $v$ and $w$. The binary variable $\mu_i$ either forces $v_i$ to be zero ($\mu_i = 0$) or forces $w_i$ to be zero ($\mu_i = 1$). In general, selecting the big-M values $M_{1, i}$ and $M_{2, i}$ is no simple task~\cite{Kleinert2020lunch}. In practice, when the big-M values are hard to obtain, we can alternatively use the special-ordered set constraint in Gurobi~\cite{gurobi} which forces constraint~\eqref{eq:LC_constr} through a branching rule instead of specifying the big-M's explicitly.

One interesting application of the formulations in~\eqref{eq:LCS} and~\eqref{eq:MIL_LCS} is in the description of the closed-loop MPC systems. We refer the readers to~\cite{simon2016stability} for the details of this formulation.

\subsection{Mixed-logical dynamical systems}
The mixed-logical dynamical system~\cite{bemporad1999control} can be written as
\begin{align}
&x_+ = A x + B_1 u + B_2 \delta + B_3 z \\
&E_1 x + E_2 u + E_3 \delta + E_4 z \leq g_5
\end{align}
where $x = [x_r^\top \ x_b^\top]^\top$ with $x_r \in \mathbb{R}^{n_r}$ and $x_b \in \{0,1\}^{n_b}$ ($u$ has a similar structure), $z \in \mathbb{R}^{n_z}$ and $\delta \in \{0, 1\}^{n_\delta}$ are auxiliary variables. The MLD system is explicitly constructed through mixed-integer linear constraints.

\subsection{ReLU neural networks}
Consider a PWA system in feedback with an $L$-layer ReLU neural network controller $u = \pi(x)$, where $\pi: \mathbb{R}^{n_x} \rightarrow \mathbb{R}^{n_u}$ is given by 
\begin{equation} \label{eq:nn_controller}
\begin{aligned}
z_0 &= x\\
z_{\ell+1} &= \max(W_\ell z_\ell + b_\ell, 0), \quad \ell = 0, \cdots, L -1 \\
\pi(x) &= W_{L} z_{L} + b_L.
\end{aligned}
\end{equation}
Here $z_0 = x \in \mathbb{R}^{n_0} \ (n_0 = n_x)$ is the input to the neural network, $z_{\ell+1} \in \mathbb{R}^{n_{\ell+1}}$ is the vector representing the output of the $(\ell+1)$-th hidden layer with $n_{\ell+1}$ neurons, $\pi(x) \in \mathbb{R}^{n_{L+1}} \ (n_{L+1} = n_u)$ is the output of the neural network, and $W_\ell \in \mathbb{R}^{n_{\ell + 1} \times n_\ell}, b_\ell \in \mathbb{R}^{n_{\ell + 1}}$ are the weight matrix and the bias vector of the $(\ell+1)$-th hidden layer, respectively.

Consider a scalar ReLU function $y=\max(0,x)$ where $\underline{x} \leq x \leq \bar{x}$. Then it can be shown that the ReLU function admits the following mixed-integer representation~\cite{tjeng2017evaluating},
\begin{equation}
\begin{aligned}
&y=\max(0,x), \ \underline{x} \leq x \leq \bar{x} \iff \\
&y \geq 0, \ y \geq x, y \leq x - \underline{x} (1-t), \ y \leq \bar{x} t, \ t \in \{0,1\},
\end{aligned}
\end{equation}
where the binary variable $t \in \{0,1\}$ is an indicator of the activation function being active ($y=x$) or inactive ($y=0$). For a ReLU network described by the equations in \eqref{eq:nn_controller}, let $\underline{m}^{\ell}$ and $\bar{m}^{\ell}$ be the element-wise lower and upper bounds on the input to the $(\ell+1)$-th activation layer, i.e., $\underline{m}_{\ell} \leq W_{\ell} z_{\ell} +b_{\ell}\leq \bar{m}_{\ell}$. Then the neural network equations are equivalent to a set of mixed-integer constraints:
\begin{equation}
\begin{aligned} \label{eq:MIL_NN}
& z_{\ell+1} = \max(W_\ell z_\ell + b_\ell,0) \iff \\ 
& \begin{cases}
z_{\ell+1} \geq W_\ell z_\ell + b_\ell \\ 
z_{\ell+1} \leq W_\ell z_\ell + b_\ell - \mathrm{diag}(\underline{m}_\ell) (\mathbf{1}-t_\ell) \\
z_{\ell+1} \geq 0 \\
z_{\ell+1} \leq \mathrm{diag}(\bar{m}_\ell)  t_\ell,
\end{cases} 
\end{aligned}
\end{equation}
where $t_{\ell} \in \{0,1\}^{n_{\ell+1}}$ is a vector of binary variables for the $(\ell+1)$-th activation layer. We note that the element-wise pre-activation bounds $\{\underline{m}_{\ell} , \bar{m}_{\ell}\}$ can be found by, for example, interval bound propagation or linear programming assuming known bounds on the input of the neural network \cite{weng2018towards,hein2017formal,wong2018provable}. Combined with the MIL formulation of PWA systems shown in Section~\ref{sec:PWA_representation}, the closed-loop dynamics of a ReLU neural network controlled PWA system admits a mixed-integer formulation.

\section{Stability Analysis via Lyapunov functions}
%
The convergence behavior of system~\eqref{eq:nonlinear_sys} around its equilibrium points can be studied by Lyapunov stability.

\begin{definition} \normalfont{(Lyapunov stability \cite[Chapter~13]{haddad2011nonlinear})} \label{def:Lyap_stability}
	The equilibrium point $x = 0$ of the autonomous system~\eqref{eq:nonlinear_sys} is
	\begin{itemize}
		\item \textbf{Lyapunov stable} if for each $\epsilon >0$, there exists $\delta = \delta(\epsilon)$ such that if 
		$\lVert x_0 \rVert < \delta$, then $\lVert x_k \rVert < \epsilon, \forall k \geq 0$.
		\item \textbf{asymptotically stable}  if it is Lyapunov stable and there exists $\delta > 0$ such that if
		$\lVert x_0 \rVert < \delta$, then $\lim_{k\rightarrow \infty} \|x_k\| = 0$.
	\end{itemize}
\end{definition}

Since the hybrid dynamics~\eqref{eq:nonlinear_sys} is nonlinear, Lyapunov stability is often a local property and it is of interest to estimate its region of attraction defined as
\begin{definition}[Region of attraction]
	The region of attraction $\mathcal{O}$ of the nonlinear system~\eqref{eq:nonlinear_sys} is the set of states from which the trajectory of system~\eqref{eq:nonlinear_sys} converges to the origin, i.e., $\mathcal{O} = \{ x_0 \in \mathcal{R} \vert \lim_{t \rightarrow \infty} x_t = 0\}$.
\end{definition}

In this work, we do not assume the domain $\mathcal{R}$ to be positive invariant. Instead, we introduce the region of interest (ROI) $\mathcal{X} \subseteq \mathcal{R}$ which is a polytopic set given by
\begin{equation}
\mathcal{X} := \{ x \in \mathbb{R}^{n_x} \vert F_\mathcal{X} x \leq h_\mathcal{X}\},
\end{equation}
to guide our search for an inner approximation $\tilde{\mathcal{O}}$ of the ROA $\mathcal{O}$. We can certify the asymptotic stability of system~\eqref{eq:nonlinear_sys} and find an $\tilde{\mathcal{O}}$ by constructing Lyapunov functions defined in the following theorem:

\begin{theorem} \label{thm:Lyapunov}
\normalfont{\cite[Chapter~13]{haddad2011nonlinear}}
	Consider the discrete-time nonlinear hybrid system~\eqref{eq:nonlinear_sys}. If there is a continuous function $V(x): \mathcal{R} \mapsto \mathbb{R}$ with domain $\mathcal{R}$ such that
	\begin{subequations}
		\label{eq:Lyap_conditions}
		\begin{align}
		\begin{split} \label{eq:Lyap_pos_constr}
		& V(0) = 0 \text{ and } V(x) > 0, \forall x \in \mathcal{X} \setminus \{0\}
		\end{split} \\
		\begin{split} \label{eq:general_Lyap_cond_3}
		& V(f(x)) - V(x) \leq 0, \forall x \in \mathcal{X},
		\end{split}
		\end{align}
	\end{subequations}
	where the set $\mathcal{X}$ is the region of interest (ROI) satisfying $\mathcal{X} \subseteq \mathcal{R}$ and $0 \in \text{int}(\mathcal{X})$, then the origin is Lyapunov stable. If, in addition, 
	\begin{equation}\label{eq:Lyap_diff_constr}
	V(f(x)) - V(x) < 0, \forall x \in \mathcal{X} \setminus \{ 0 \},
	\end{equation}
	then the origin is asymptotically stable. 
\end{theorem}
We call any $V(\cdot)$ satisfying~\eqref{eq:Lyap_pos_constr} a Lyapunov function \emph{candidate}. If additionally, $V(\cdot)$ satisfies the condition~\eqref{eq:general_Lyap_cond_3} or~\eqref{eq:Lyap_diff_constr}, then $V(\cdot)$ is called a \emph{valid} Lyapunov function candidate, or simply, a Lyapunov function. Since asymptotic stability is our primary focus in this paper, Lyapunov functions refer to any $V(\cdot)$ satisfying constraints~\eqref{eq:Lyap_pos_constr} and~\eqref{eq:Lyap_diff_constr} unless specified otherwise. Once a Lyapunov function $V(x)$ is obtained, an inner estimate of the ROA is given by 
$\mathcal{\tilde{O}} = \{ x \vert V(x) \leq \tau\}$, 
where $\tau = \inf_{x \in \mathcal{R} \setminus \mathcal{X}} V(x)$. In other words, $\mathcal{\tilde{O}} \subset \mathcal{X}$ is the largest sublevel set of $V(x)$ that is contained in $\mathcal{X}$. 

\begin{remark}
The ROI $\mathcal{X}$ can be set according to prior knowledge or from simulation of the nonlinear hybrid dynamics~\eqref{eq:nonlinear_sys}. See Section~\ref{sec:simulation} for examples of choosing the ROI.
\end{remark}

\subsection{Lyapunov function parameterization}
Searching for a Lyapunov function in the function space is intractable since the problem is infinite-dimensional in this space. Instead, we reduce our search space to the class of Lyapunov functions defined by 
\begin{equation} \label{eq: multi-step Lyapunov}
V^{(k)}(x;P) = {z^{(k)}}^\top P z^{(k)},
\end{equation}
where $P \in \mathbb{S}^{(k+1)n_x}_{++}$, and $z^{(k)}$ is given by
\begin{equation} \label{eq:basis}
z^{(k)} = [x^\top  \quad f(x)^{(1),\top} \quad f^{(2), \top}(x) \quad \cdots \quad f^{(k), \top}(x)]^\top
\end{equation}
Here we use the notation $f^{(0)}(x) = x, f^{(k+1)} = f(f^{(k)}(x)), \forall k \geq 0$. We call $V^{(k)}(x;P)$ the Lyapunov function candidate of order $k$, which is a quadratic function in composition with the system dynamics $f(x)$ evolved for $k$ steps. Indeed, we can increase the complexity of the function class monotonically by increasing the order $k$.

Since $P$ is positive definite, searching for a valid Lyapunov function of the form \eqref{eq: multi-step Lyapunov} reduces to find a matrix $P$ that satisfies the Lyapunov difference condition ~\eqref{eq:Lyap_diff_constr}. Explicitly, we can characterize the space of matrices $P$ that admit a valid Lyapunov function candidate as
\begin{equation} \label{eq:target_set}
\mathcal{F} \!= \!\{P \in \mathbb{S}^{(k+1)n_x} \vert \alpha I \! \preceq P \! \preceq \!\beta I, \Delta V^{(k)}(x, P) < 0, \forall x \in \mathcal{X} \setminus \{0\} \}.
\end{equation} 
where $0 \leq \alpha < \beta$, and $\Delta V^{(k)}(x, P)$ is the Lyapunov difference given by
\begin{equation}
\Delta V^{(k)}(x, P) := V^{(k)}(f(x);P) - V^{(k)}(x;P).
\end{equation}
The constraint $\alpha I \preceq P$ guarantees  the condition~\eqref{eq:Lyap_pos_constr}, while the constraint $P \preceq \beta I$ is imposed to make $\mathcal{F}$ bounded without loss of generality since we can always scale $P$ while satisfying~\eqref{eq:Lyap_diff_constr}. 

We call $\mathcal{F}$ the target set. It follows that $\mathcal{F}$ is convex since it is defined by semidefinite as well as linear constraints on $P$. As $\alpha I \preceq P \preceq \beta I$ imposes an additional constraint on the condition number of $P$, in practice we choose $\beta/\alpha$ large or simply set $\alpha = 0$~\footnote{As will be shown next, the proposed method always finds a feasible solution in the interior of $\mathcal{F}$. Therefore, choosing $\alpha = 0$ does not affect the positive definiteness of the solution.}. Then finding a Lyapunov function in a given function class $V^{(k)}(x;P)$ is stated as the following problem:

\begin{problem}\label{prob:target_set}
	For each parameterized function class $V^{(k)}(x;P)$ and the target set $\mathcal{F}$ defined in~\eqref{eq:target_set}, find a feasible point in $\mathcal{F}$ or certify that $\mathcal{F}$ is empty.
\end{problem}

Although the target set $\mathcal{F}$ is convex, the fact that it is characterized by infinitely many linear constraints (i.e., the constraints $\Delta V^{(k)}(x, P)$ $<0, \forall x \in \mathcal{X} \setminus \{0\}$) poses computational challenges to solving Problem~\ref{prob:target_set}. In this work, we propose a learning-based approach to address this challenge by iteratively drawing state samples to refine our over-approximation of the target set $\mathcal{F}$. By designing the learning strategy based on ACCPM from convex optimization, we show that when the target set $\mathcal{F}$ is full-dimensional in the parameter space of $P$, our method is guaranteed to find a feasible point in $\mathcal{F}$ in a finite number of steps.

\begin{remark}
	The parameterization of $V^{(k)}(x;P)$ is inspired by the finite-step Lyapunov function~\cite{aeyels1998new, bobiti2016sampling} and the non-monotonic Lyapunov functions~\cite{ahmadi2008non} which also construct Lyapunov function candidates using the system states several steps ahead. $V^{(k)}(x;P)$ allows us to parameterize complex function classes with a relatively small number of parameters. For example, when $f(x)$ is a PWA function with $N_{mode}$ (possibly large) partitions, $V^{(1)}(x;P)$ is a PWQ function with the same partitions in the state space.	
\end{remark}

\section{Learning Lyapunov functions from counterexamples}
\label{sec:ACCPM}

We recall from the previous section that finding a feasible point of the convex set $\mathcal{F}$ is computationally intractable since the condition $\Delta V^{(k)}(x;P) < 0$ must hold for all $x \in \mathcal{X} \setminus \{0\}$. To overcome this intractability, we adopt a learning-based approach, in which we first select a set of finite samples $\mathcal{S} = \{x^1, x^2, \cdots, x^N\} \subset \mathcal{X}$ and then enforce the linear constraint $\Delta V^{(k)}(x;P) < 0$ to hold only for $x \in \mathcal{S}$. This results in an over-approximation of $\mathcal{F}$ given by

\begin{equation} \label{eq:localization}
	\tilde{\mathcal{F}} = \{ P \in \mathbb{S}^{(k+1)n_x} \vert \alpha I \preceq P \preceq \beta I, \Delta V^{(k)}(x, P) \leq 0, \forall x \in \mathcal{S}\}.
\end{equation} 
We call $\tilde{\mathcal{F}}$ the localization set. Finding a feasible point in $\tilde{\mathcal{F}}$ now becomes a tractable convex feasibility problem. However, there is no guarantee that a $P \in \tilde{\mathcal{F}}$ would correspond to a valid Lyapunov function, even if the number of samples approaches infinity. As one of our contributions, we propose an efficient learning strategy based on the analytic center cutting-plane method (ACCPM) to iteratively grow the sample set and refine the set $\tilde{\mathcal{F}}$ until we find a feasible point in $\mathcal{F}$. In the next subsections, we describe the proposed approach and provide finite-step termination guarantees when $\mathcal{F}$ is non-empty and satisfies the following assumption:

\begin{assumption} \label{assump:nonempty}
The target set $\mathcal{F}$ defined in~\eqref{eq:target_set} is full-dimensional and there exists $P_{center}\in \mathbb{S}^{(k+1)n_x}$ such that $\{P \in \mathbb{S}^{(k+1)n_x} \vert \lVert P - P_{center} \rVert_F \leq \epsilon \} \subset \mathcal{F}$ where $\lVert \cdot \rVert_F$ is the Frobenius norm. 
\end{assumption} 
 
\subsection{Analytic center cutting-plane method}
Cutting-plane methods~\cite{atkinson1995cutting, elzinga1975central,boyd2007localization} are iterative algorithms to find a point in a target convex set $\mathcal{F}$ or to determine whether $\mathcal{F}$ is empty. In these methods, we have no information on $\mathcal{F}$ except for a ``cutting-plane oracle'', which can verify whether $P \in \mathcal{F}$ for a given $P$. Let $\tilde{\mathcal{F}}$ be a localization set defined by a finite set of inequalities that over approximates the target set, i.e., $\mathcal{F} \subseteq \tilde{\mathcal{F}}$. If $\tilde{\mathcal{F}}$ is empty, then we have proof that the target set $\mathcal{F}$ is also empty. Otherwise, we query the oracle at a point $P \in \tilde{\mathcal{F}}$. If $P \in \mathcal{F}$, the oracle returns `yes' and the algorithm terminates; if $P \notin \mathcal{F}$, it returns `no' together with a separating hyperplane that separates $P$ and $\mathcal{F}$. In the latter case, the cutting-plane method updates the localization set by $\tilde{\mathcal{F}} \leftarrow  \tilde{\mathcal{F}} \cap \{\text{half space defined by the separating hyperplane}\}$ as shown in Fig.~\ref{fig:cutting_plane_demo}. This process continues until either a point in the target set is found or the target set is certified to be empty.

Based on how the query point is chosen, different cutting-plane methods have been proposed including the center of gravity method \cite{Lev65}, the maximum volume ellipsoid (MVE) cutting-plane method \cite{Tarasov1988}, the Chebyshev center cutting-plane method \cite{elzinga1975central}, the ellipsoid method \cite{khachiyan1980polynomial, ellipsoidYudin}, and the analytic center cutting-plane method \cite{goffin1993computation, nesterov1995cutting, atkinson1995cutting}. In this paper, we use the analytic center cutting-plane method since it allows the localization set $\tilde{\mathcal{F}}$ to be described by linear matrix inequalities. In the ACCPM, the query point $P$ is chosen as the analytic center of the localization set $\tilde{\mathcal{F}}$.


\begin{figure}
	\centering
	\includegraphics[width = 0.6 \textwidth]{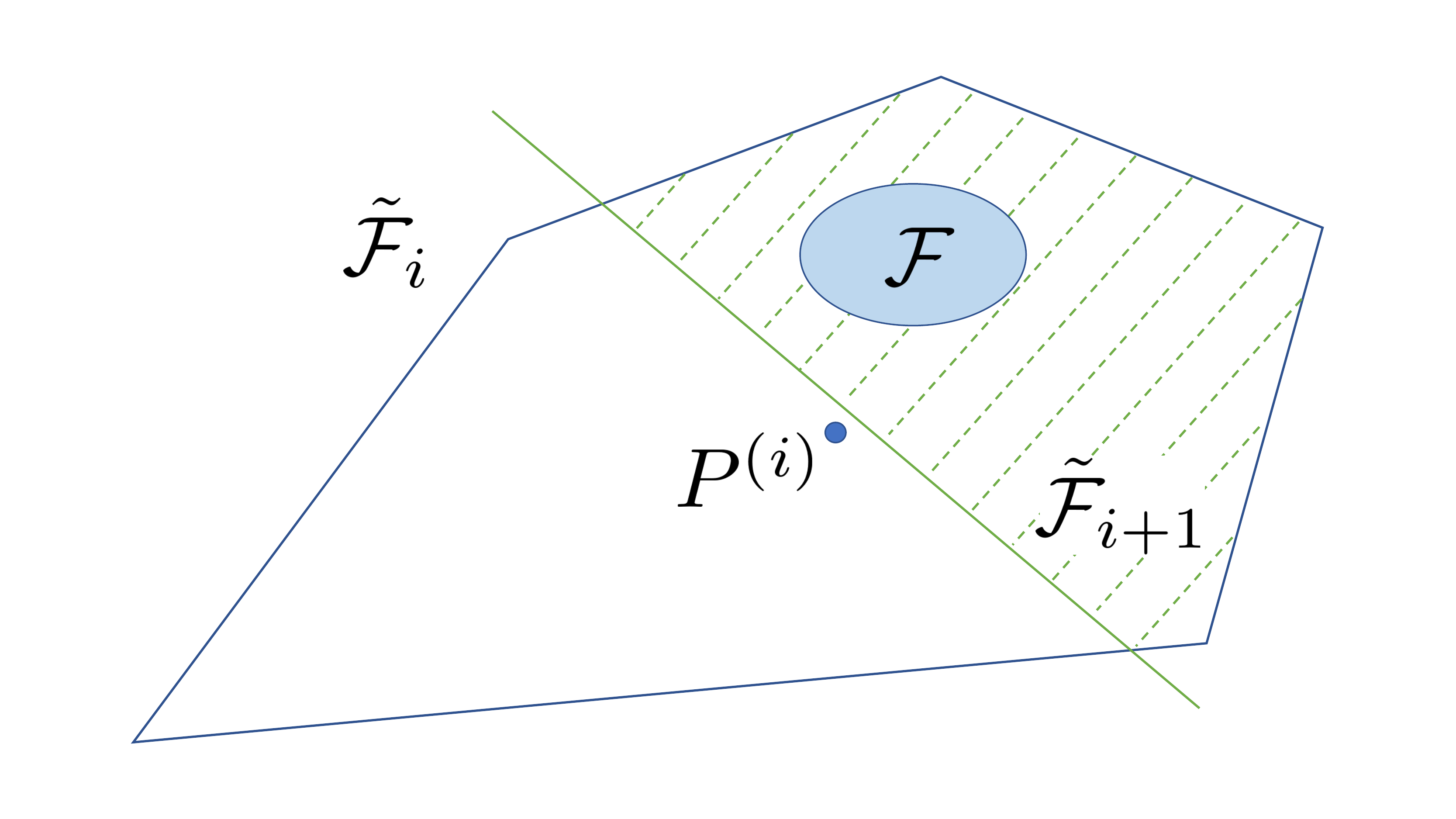}	
	\caption{In the $i$-th iteration, the query point $P^{(i)}$ is chosen as the ``center'' of the localization set $\tilde{\mathcal{F}}_i$ to guarantee the removal of a portion of $\tilde{\mathcal{F}}_i$ from the search space whenever a separating hyperplane (green solid line) is given. The localization set is then updated to $\tilde{\mathcal{F}}_{i+1}$ (shaded set) which is a finer over-approximation of the target set $\mathcal{F}$.}
	\label{fig:cutting_plane_demo}
\end{figure}

\begin{definition}[Analytic center~\cite{boyd2004convex}] \label{def:analytic_center}
	The analytic center $x_{ac}$ of a set of convex inequalities and linear equalities
	$h_i(x) \leq 0, i = 1,\! \cdots\!, m, \ Fx = g$,
	is defined as the solution of the convex problem
	\begin{equation} \label{eq:analytic_center}
	\begin{aligned}
	\underset{x}{\mathrm{minimize}} & \quad - \sum_{i=1}^m  \log(-h_i(x))  \quad \text{subject to} \quad Fx = g.
	\end{aligned}
	\end{equation}
\end{definition}

We design the learning strategy according to the ACCPM by constructing a learner,  which proposes Lyapunov function candidates based on a set of samples, and a verifier, which serves as a cutting-plane oracle and updates the sample set with counterexamples.


\subsection{The learner}
Let $\mathcal{S} = \{x^1, x^2, \cdots, x^N \} \subset \mathcal{X}$ be a collection of samples from $\mathcal{X}$. The localization set $\tilde{\mathcal{F}}$ in the space of $P$ is given by~\eqref{eq:localization},
which represents the learner's knowledge about $\mathcal{F}$ by observing the $k$-step trajectories of the dynamical system~\eqref{eq:nonlinear_sys} starting from $\mathcal{S}$. According to the ACCPM, the learner proposes a Lyapunov function candidate $V^{(k)}(x; P_{ac})$ with $P_{ac}$ as the analytic center of $\tilde{\mathcal{F}}$:

\begin{equation}\label{eq:ac_optimization}
\begin{aligned}
P_{ac}: = \underset{P}{\text{argmin}} & \quad -\sum_{x\in \mathcal{S}} \log( -\Delta V^{(k)}(x,P) ) \\
&\quad \quad \quad  - \log\det (\beta I - P) - \log \det (P -  \alpha I).
\end{aligned}
\end{equation}
This is a convex program that can be solved efficiently through off-the-shell convex optimization solvers. We denote the objective function in~\eqref{eq:ac_optimization} as $\phi(\tilde{\mathcal{F}})$ and call it the potential function on the set $\tilde{\mathcal{F}}$. If \eqref{eq:ac_optimization} is infeasible, then we have a proof that no Lyapunov function exists in the function class of order $k$. Otherwise, the learner proposes $V^{(k)}(x; P_{ac}) = {z^{(k)}}^\top P_{ac} z^{(k)}$ as a Lyapunov function candidate. Due to the log-barrier function in~\eqref{eq:ac_optimization}, $P_{ac}$ is in the interior of $\tilde{\mathcal{F}}$. When the sample set $\mathcal{S}$ is empty, the potential function simply becomes $\phi(\tilde{\mathcal{F}}) = -\log \det(\beta I - P) - \log \det (P - \alpha I)$ with $P_{ac} = \frac{\alpha + \beta}{2}I$ as the optimal solution. 

\subsection{The verifier}
Suppose the learner proposes the Lyapunov function candidate $V^{(k)}(x; P)$ by solving~\eqref{eq:ac_optimization}. Given $V^{(k)}(x; P)$, the verifier either ensures that this function satisfies the constraints in~\eqref{eq:Lyap_pos_constr} and~\eqref{eq:Lyap_diff_constr}, or returns a state where constraint~\eqref{eq:Lyap_pos_constr} or~\eqref{eq:Lyap_diff_constr} is violated as a counterexample. Since the log-barrier function in~\eqref{eq:ac_optimization} guarantees $P \succ 0$, constraint~\eqref{eq:Lyap_pos_constr} is readily satisfied and the verifier must check the violation of constraint~\eqref{eq:Lyap_diff_constr}. This can be done by solving the optimization problem

\begin{alignat}{2} \label{eq:verifier_formulation}
&\underset{x \in \mathcal{X} \setminus \{0\}}{\mathrm{maximize}} \quad && \Delta V^{(k)}(x,P).
\end{alignat}


When $f(x)$ has a mixed-integer formulation
\begin{equation*}
	\mathcal{L}_{\text{gr}(f)} = \{(x, x_+, \lambda, \mu) \in \mathbb{R}^{n_x} \times \mathbb{R}^{n_x} \times \mathbb{R}^{n_\lambda}\times \mathbb{Z}^{n_\mu} \vert \ell(x, x_+, \lambda, \mu) \leq v \},
\end{equation*} 
for Lyapunov function candidates $V^{(k)}(x;P)$ of order $k$, we write problem~\eqref{eq:verifier_formulation} explicitly as a mixed-integer quadratic program (MIQP):

\begin{subequations} \label{eq:MIQP_f}
	\begin{align}
	\begin{split}\label{eq:MIQP_f_obj}
	\underset{ \{x_i\}, \{\lambda_i\}, \{\mu_i\} }{\text{maximize}} & \quad  \begin{bmatrix}
	x_1 \\ x_2 \\ \vdots \\ x_{k+1}
	\end{bmatrix}^\top P \begin{bmatrix}
	x_1 \\ x_2 \\ \vdots \\ x_{k+1}
	\end{bmatrix} - 
	\begin{bmatrix}
	x_0 \\ x_1 \\ \vdots \\ x_{k}
	\end{bmatrix}^\top P \begin{bmatrix}
	x_0 \\ x_1 \\ \vdots \\ x_{k}
	\end{bmatrix} 
	\end{split} \\
	\begin{split} \label{eq:MIQP_f_ROI}
	\text{subject to} & \quad x_0 \in \mathcal{X}
	\end{split} \\
	\begin{split} \label{eq:MIQP_f_exclusion}
	& \quad \lVert x_0 \rVert_\infty \geq \epsilon
	\end{split}\\
	\begin{split} \label{eq:MIQP_f_MIL}
	& \quad \ell(x_i, x_{i+1}, \lambda_i, \mu_i) \leq v_i, i = 0, 1, \cdots, k
	\end{split}
	\end{align}
\end{subequations}
where the objective function~\eqref{eq:MIQP_f_obj} is equal to the Lyapunov difference $\Delta V^{(k)}(x_0,P)$, constraint~\eqref{eq:MIQP_f_ROI} restricts the search of counterexamples inside the ROI $\mathcal{X}$, constraint~\eqref{eq:MIQP_f_exclusion} excludes a small $\ell_\infty$ norm ball centered at the origin $B_\epsilon = \{ x \vert \lVert x \rVert_\infty < \epsilon \}$ from the search space, and constraints~\eqref{eq:MIQP_f_MIL} enforce the dynamical constraint $(x_i, x_{i+1}) \in \text{gr}(f)$, or equivalently $x_{i+1}= f(x_i)$ for $i =0, \cdots, k$. 

Denote $p^*$ the optimal value and $x_0^*$ the $x_0$-component of the optimal solution of~\eqref{eq:MIQP_f}. When $p^* \geq 0$, $x_0^*$ is a counterexample for the Lyapunov function candidate $V^{(k)}(x;P)$ since $\Delta V^{(k)}(x_0^*, P) \geq 0$. Then, the separating hyperplane induced by the counterexample $x_0^*$ is given as $\Delta V^{(k)}(x_0^*, P) = 0$ where $\Delta V^{(k)}(x_0^*, P)$ is interpreted as a linear function in the matrix variable $P$. When $p^* < 0 $, we certify the convergence of system trajectories to a neighborhood of the origin as shown in the following corollary.
\begin{corollary} \label{coro:convergence}
	Let $\Omega = \{ x \vert \allowdisplaybreaks V^{(k)}(x;P) \leq \allowdisplaybreaks \gamma\}$ be the largest sublevel set of $V^{(k)}(x;P)$ inside $\mathcal{X}$ with $\gamma = \inf_{x \in \mathcal{R}\setminus \mathcal{X}} V^{(k)}(x;P)$. Define the successor set of $B_\epsilon$ as $\text{suc}(B_\epsilon):=\{y\vert y = f(x), x \in B_\epsilon \}$. Assume $B_\epsilon \subset \Omega$, $\text{suc}(B_\epsilon) \subset \Omega$, and $\Omega_{loc}$ is the smallest sublevel set of $V^{(k)}(x;P)$ such that $\text{suc}(B_\epsilon) \subseteq \Omega_{loc}$. If $p^* < 0$, then we have $\lim_{t \rightarrow \infty} x_t \in \Omega_{loc}$ for all $x_0 \in \Omega$.
\end{corollary}

\begin{proof}
	First, note that all trajectories starting from $x_0 \in \Omega$ and $x_0 \notin B_\epsilon$ reach $B_\epsilon$ in a finite number of steps since $V^{(k)}(x_{t+1};P) - V^{(k)}(x_t; P) \leq p^* < 0$ as long as $x_t \notin B_\epsilon$ and $V^{(k)}(x_t; P)$ is lower-bounded by $0$ for all $x_t$. If $x_t$ never reaches $B_\epsilon$, we will have a contradiction that $V^{(k)}(x_N; P) < 0$ for some $N$. After the trajectories reach $B_\epsilon$, the subsequent states will remain inside $\Omega_{loc}$ by construction. Therefore, we have $\lim_{t \rightarrow \infty} x_t \in \Omega_{loc}$ for all $x_0 \in \Omega$.
\end{proof}

The alternation between the learner and the verifier is summarized in Algorithm~\ref{alg:ACCPM}. In the next subsection, we show that when the target set $\mathcal{F}$ satisfies Assumption~\ref{assump:nonempty}, our proposed algorithm is guaranteed to find a feasible point in $\mathcal{F}$ in a finite number of steps. 

\begin{algorithm}[htb]
\caption{Learning-based Lyapunov function synthesis}\label{alg:ACCPM}
	\begin{algorithmic}[1]
		\Procedure{LearningLyapunov}{}
		\State $\mathcal{S}_1 = \emptyset$ \Comment{initialize the sample set}
		\State $i=1$
		\While {True}
		\State Generate $\tilde{\mathcal{F}}_i$ from sample set $\mathcal{S}_i$ by~\eqref{eq:localization}.
		\If {$\tilde{\mathcal{F}}_{i} = \emptyset$} 
		\State Return: Status = Infeasible, $P_* = \emptyset$.
		\EndIf	
		\State $P^{(i)} = \arg \min \eqref{eq:ac_optimization} \text{ with } $$\mathcal{S}_i$  \Comment{find analytic center}
		\State solve~\eqref{eq:MIQP_f} with $P^{(i)}$  \Comment{query the verifier}
		\If {$\max \ $\eqref{eq:MIQP_f} < 0} 
		\State Return: Status = Feasible, $P_* = P^{(i)}$.
		\Else  
		\State  $x_*^{(i)} = \arg \min_{x_0} \eqref{eq:MIQP_f}$ \Comment{find a counterexample}
		\State $\mathcal{S}_{i+1} = \mathcal{S}_i \cup \{x^{(i)}_*\}$
		\EndIf	
		\State $i = i+1$
		\EndWhile
		\EndProcedure
	\end{algorithmic}
\end{algorithm}

\begin{remark}
In this paper, the verifier is constructed as a global optimization problem~\eqref{eq:verifier_formulation}. Alternative constructions of the verifier do exist, e.g., by using SMT solvers as shown in~\cite{ahmed2020automated, abate2020formal}. The termination results in the next subsection will always hold no matter how the verifier is formulated.  
\end{remark}

\subsection{Convergence Analysis}
The convergence and complexity of the ACCPM have been studied in~\cite{atkinson1995cutting,nesterov1995cutting, ye1992potential, luo2000polynomial,goffin1996complexity, sun2002analytic} under various assumptions on the localization set, the form of the separating hyperplane, whether multiple cuts are applied, etc. Directly related to Algorithm~\ref{alg:ACCPM} and the search for $V^{(k)}(x;P)$ is~\cite{sun2002analytic} which analyzes the complexity of the ACCPM with a matrix variable and semidefiniteness constraints. Notably, it provides an upper bound on the number of iterations that Algorithm~\ref{alg:ACCPM} can run before termination when the target set is non-empty. In~\cite{sun2002analytic}, it is assumed that
\begin{itemize}
	\item A1: $\mathcal{F}$ is a convex subset of $\mathbb{S}^{n}$.
	\item A2: See Assumption~\ref{assump:nonempty}.
	\item A3: $\mathcal{F} \subset \{P \in \mathbb{S}^{n} \vert 0 \preceq P \preceq I \}$.
\end{itemize}
For the Lyapunov function candidate class of order $k$, we have $n = (k+1)n_x$. Let the ACCPM start with the localization set $\tilde{\mathcal{F}}_1 = \{ P \vert 0 \preceq P \preceq I\}$ and initialize the first query point $P^{(1)} = \frac{1}{2}I$ correspondingly. If at iteration $i$ a query point $P^{(i)}$ is rejected by the oracle, a separating hyperplane of the form $\langle D_i, P - P^{(i)} \rangle = 0$ with $\lVert D_i \rVert_F = 1$ is given by the verifier. By induction, we have that at iteration $i > 1$, the localization set $\tilde{\mathcal{F}_{i}}$ is 
\begin{equation*}
\tilde{\mathcal{F}_{i}} = \{ P \vert 0 \preceq P \preceq I, \langle D_j, P \rangle \leq c_j, j = 1, \cdots, i-1 \}.
\end{equation*}
with $D_j$ defining the separating hyperplane and $c_j = \langle D_j, P^{(j)} \rangle$. The query point at iteration $i$ is given by $P^{(i)} =  \arg\min \phi(\tilde{\mathcal{F}_{i}})$
with the potential function
\begin{equation*}
\phi(\tilde{\mathcal{F}_{i}}) = -\sum_{j = 1}^{i-1} \log(c_j - \langle D_j, P\rangle) - \log \det (I -P) - \log \det(P).
\end{equation*}

\begin{theorem}
\label{thm:ACCPM_termination}
Under Assumption~\ref{assump:nonempty} on the target set $\mathcal{F}$, Algorithm~\ref{alg:ACCPM} finds a feasible point in~$\mathcal{F}$ in at most $O(((k+1)n_x)^3/\epsilon^2)$ iterations.
\end{theorem} 

\begin{proof}
The proof follows from~\cite{sun2002analytic} which states for a general target set $\mathcal{F}\subset \mathbb{S}^{n}$ under assumptions A$1$ to A$3$, the analytic center cutting-plane method with separating hyperplanes of the form $\langle D_i, P \rangle \leq c_i$ is shown to find a feasible point in at most $O(n^3/\epsilon^2)$ iterations. For the sequence of localization sets $\tilde{\mathcal{F}}_i$, \cite{sun2002analytic} computes an upper bound on the potential function $\phi(\tilde{\mathcal{F}}_i)$ which is approximately $i \log( \frac{1}{\epsilon})$, and a lower bound on $\phi(\tilde{\mathcal{F}}_i)$ which is proportional to $\frac{i}{2}\log(\frac{i}{n^3})$. Since the ACCPM must terminate before the lower bound exceeds the upper bound of $\phi(\tilde{\mathcal{F}}_i)$, we obtain the $O(n^3/\epsilon^2)$ upper bound on the number of iterations.

To show that the result in~\cite{sun2002analytic} applies to Algorithm~\ref{alg:ACCPM}, note that for each counterexample $x_*^{(i)}$ found by the verifier in iteration $i$, the separating hyperplane can be constructed as $\Delta V^{(k)}(x_*^{(i)}, P) \leq \Delta V^{(k)}(x_*^{(i)},P^{(i)})$ since $\Delta V^{(k)}(x_*^{(i)},P^{(i)}) \geq 0$. We can rewrite this cutting plane in the form $\langle D_i, P \rangle \leq \langle D_i, P^{(i)} \rangle$ by setting $\hat{D}_i = z_{k,*}^{+, (i)} z_{k,*}^{+,(i), \top} - z_{k,*}^{(i)} z_{k,*}^{(i), \top}, D_i = \hat{D}_i / \lVert \hat{D}_i \rVert_F$, where $z_{k,*}^{(i)}$ denotes the basis~\eqref{eq:basis} with $x = x_*^{(i)}$, and $z_{k,*}^{+, (i)}$ denotes the basis~\eqref{eq:basis} after setting $x = f(x_*^{(i)})$. Then with $\alpha = 0, \beta = 1$ in~\eqref{eq:target_set}, Algorithm~\ref{alg:ACCPM} satisfies assumptions A$1$ to A$3$ and it terminates in at most $O(((k+1)n_x)^3/\epsilon^2)$ iterations according to~\cite{sun2002analytic}.
\end{proof}


Theorem~\ref{thm:ACCPM_termination} provides a finite-step termination guarantee for Algorithm~\ref{alg:ACCPM} when the target set satisfies Assumption~\ref{assump:nonempty}. However, when $\mathcal{F}$ is empty, we do not have such a guarantee. Certification of the non-existence of Lyapunov functions in $V^{(k)}(x;P)$ relies on detecting that an over-approximation $\tilde{\mathcal{F}}$ is empty. Hence, a quick expansion of the sample set $\mathcal{S}$ as shown in Section~\ref{sec:early_termination} is preferred.


\subsection{Implementation}

\subsubsection{Complexity of the learner}
For a sample set $\mathcal{S}$ with $N$ samples, the localization set $\tilde{\mathcal{F}}$ in ~\eqref{eq:localization} is described by $N$ linear as well as two semidefinite constraints on the variable $P \in \mathbb{S}^{(k+1)n_x}$. We first decide if $\tilde{\mathcal{F}}$ is empty by solving an SDP feasibility problem. If the SDP is infeasible, then $\tilde{\mathcal{F}} = \emptyset$ and so is $\mathcal{F}$. If $\mathcal{\mathcal{F}}$ is non-empty, we move on to the analytic center problem~\eqref{eq:ac_optimization} which can be solved, e.g., through an infeasible start Newton's method~\cite{boyd2004convex}.


\subsubsection{Complexity of the MIQP}
Since MIQP is well-known to be NP-hard, we use the number of binary variables, which we denote by $N_{\mu}$ in~\eqref{eq:MIQP_f}, as a rough measure of the complexity of~\eqref{eq:MIQP_f}. When $f(x)$ is a PWA function with $N_{mode}$ modes, we have $N_{mode}$ binary variables in the MIL formulation of $f(x)$. For the LC system~\eqref{eq:LCS}, $N_\mu$ equals the dimension of the orthogonal variables $v$ and $w$. $N_\mu$ is explicitly given in the MLD system and is equal to the number of neurons in the mixed-integer formulation of the ReLU networks. It follows that for the LC systems and ReLU networks, it is possible to use a small number of integer variables to encode a PWA system with many more modes. Since we need to evaluate the hybrid dynamics $f(x)$ for $k$ times when $V^{(k)}(x;P)$ is applied, the number of binary variables $N_\mu$ in~\eqref{eq:MIQP_f} is largely linear in the order $k$.

The actual solving time of the MIQP has a complex dependence not only on the number of variables and constraints but also on how the constraints are formulated. The exploration of the numerical performance of the proposed algorithm is left for future research.

\subsubsection{Solvability of the MIQP}
The optimization problem~\eqref{eq:MIQP_f} is a nonconvex MIQP since the quadratic objective function is indefinite. Therefore, the relaxation of the problem after removing the integrality constraints would result in a nonconvex quadratic program. Nonconvex MIQP can be solved to global optimality through Gurobi v$9.0$~\cite{gurobi} by transforming the nonconvex quadratic expression into a bilinear form and applying spatial branching~\cite{belotti2013mixed}. More information on solving nonconvex mixed-integer nonlinear programming can be found in~\cite{vigerske2013decomposition, tawarmalani2013convexification, belotti2009branching}. In this paper, we rely on Gurobi to solve the nonconvex MIQP~\eqref{eq:MIQP_f} automatically.

\subsubsection{Exclusion of the origin}
In constraint~\eqref{eq:MIQP_f_exclusion}, we add a guard $B_\epsilon$ at the origin to approximate the exact constraint $x\in \mathcal{X} \setminus \{0\}$. Since constraint~\eqref{eq:MIQP_f_exclusion} allows a mixed-integer linear formulation through the big-M method, problem~\eqref{eq:MIQP_f} is an MIQP. Adding $B_\epsilon$ bounds $p^*$ off from $0$ if $V_k(x;P)$ is in fact a Lyapunov function and we can decide the negativity of $p^*$ by checking if $p^* < -\epsilon_{tol}$ for some tolerance $\epsilon_{tol} > 0$ to handle round-off errors in computation. 

When the dynamics $x_+ = f(x)$ is linear inside $B_\epsilon$, i.e., $x_+ = Ax$ for $x \in B_\epsilon$, we can show convergence to the origin of the system trajectories starting inside $B_\epsilon$ by checking the magnitude of the eigenvalues of $A$. Combined with Corollary~\ref{coro:convergence}, asymptotic convergence inside the sublevel set $\Omega$ can be established.

\subsubsection{Early termination of the MIQP}
\label{sec:early_termination}
We can terminate the Branch $\&$ Bound~\cite{wolsey1999integer} algorithm in solving~\eqref{eq:MIQP_f} once a feasible solution is found with non-negative objective function since in this case we already have a counterexample.

\section{Numerical examples}
\label{sec:simulation}
We demonstrate our method through two examples: closed-loop MPC systems and ReLU neural network controlled PWA systems. In particular, we compare the performances of our method with the LMI-based Lyapunov function synthesis methods~\cite{biswas2005survey} on the MPC example. Algorithm~\ref{alg:ACCPM} is implemented in Python $3.7$ with Gurobi v9.0~\cite{gurobi} and the LMI-based method is implemented in the MPT3 toolbox~\cite{MPT3} with Mosek~\cite{mosek} in Matlab. All the simulation is implemented on an Intel i7-6700K CPU with $32$ GB of RAM.

Throughout the numerical experiments in this section, Algorithm~\ref{alg:ACCPM} is run with $\alpha = 0.0, \beta = 1.0$ and $\epsilon_{tol} = 10^{-8}$ to decide negativity of the optimal value of the MIQP. The Gurobi solver is set with feasibility tolerance $10^{-9}$, integer tolerance $10^{-9}$, and optimality tolerance $10^{-9}$. In addition, we terminate the MIQP once it finds a feasible solution that generates an objective $\geq 10^{-4}$ which means a counterexample is already found.

\subsection{A 2-dimensional closed-loop MPC system}
\label{sec:2d_example}
For an unstable linear system $x_+ = Ax + Bu$ with 
\begin{equation} \label{eq:2d_dyn}
A = \begin{bmatrix}
1.2 & 1.2 \\ 0 & 1.2
\end{bmatrix},\quad  B = \begin{bmatrix}
1 \\ 0.5 \end{bmatrix},
\end{equation}
we design an MPC controller with horizon $T = 10$, state constraint $[-5 \ -5]^\top \leq x \leq [5 \ 5]^\top$, control input constraint $-1 \leq u \leq 1$, stage cost $p(x, u) = x^\top Q x  + u^\top R u$ with $Q = 10 I, R = 1$, and terminal cost $q(x) = x^\top P_\infty x$ where $P_\infty = \text{DARE}(A, B, Q, R)$ is the solution to the discrete algebraic Riccati equation defined by $(A, B, Q, R)$. The terminal set is chosen as the maximum positive invariant set~\cite[Chapter 10]{borrelli2017predictive} of the closed-loop system $x_+ = (A + B K_\infty) x$ where $K_\infty = -(B^\top P_\infty B + R)^{-1}B^\top P_\infty A$.

\begin{figure}
	\centering
	\includegraphics[width = 0.5\textwidth]{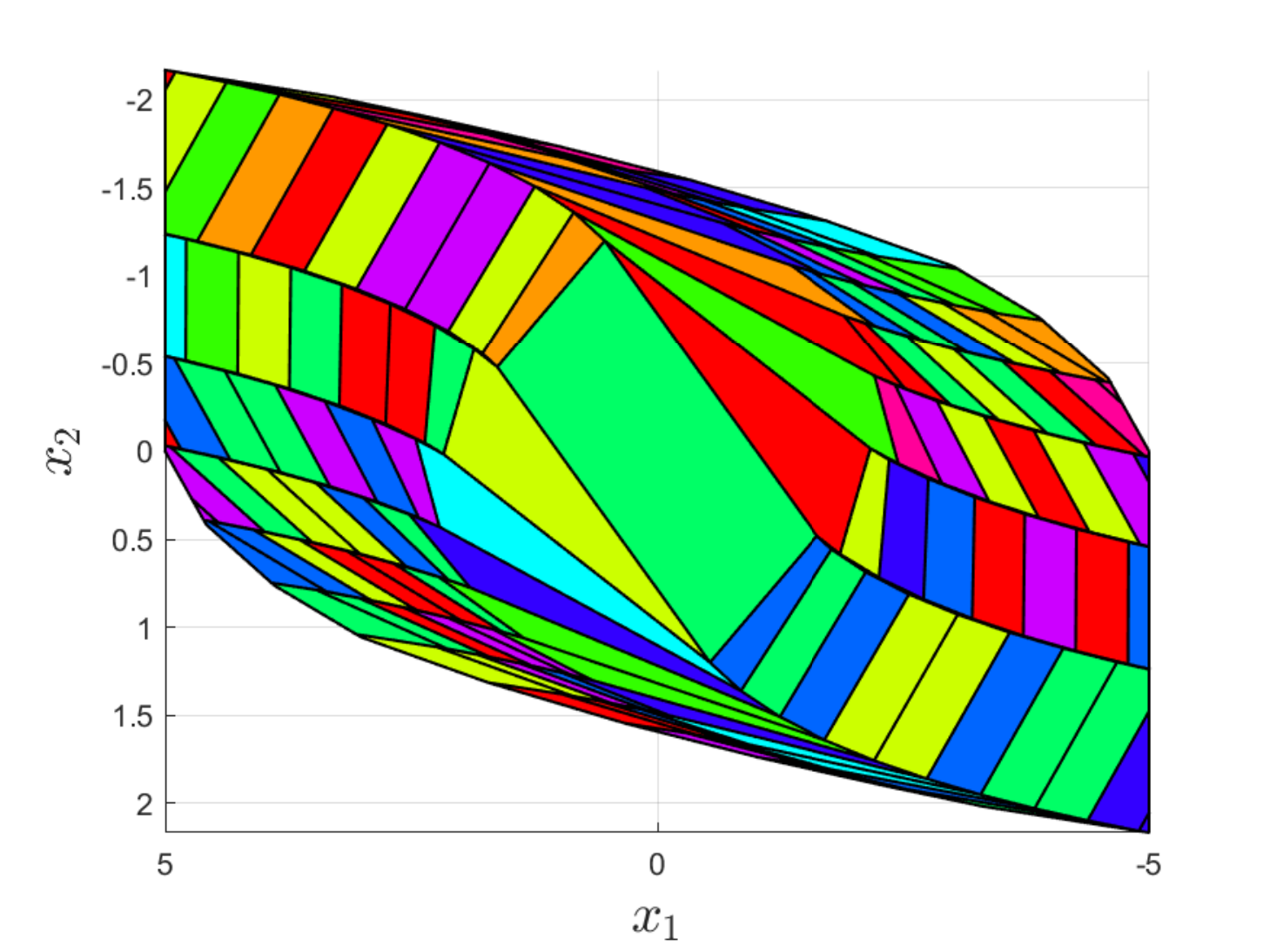}
	\caption{Domain $\mathcal{R}$ and the $211$ modes of the closed-loop MPC system of~\eqref{eq:2d_dyn}.}
	\label{fig:MPC_2d_domain}
\end{figure}

With the above setup, we obtain the explicit MPC controller, which is a PWA function of the state $x$ with $211$ partitions, through the MPT3 toolbox in Matlab. The domain $\mathcal{R}$ of the explicit MPC is a polytope shown in Fig.~\ref{fig:MPC_2d_domain} together with its partitions. After obtaining $\mathcal{R}$, we verify that it is positive invariant for the closed-loop dynamics. Through the LMI-based method, we are able to synthesize a discontinuous PWQ Lyapunov function in MPT3 which verifies $\mathcal{R}$ is the ROA for the closed-loop MPC system. The total running time is $38.890$ seconds, with $0.25$ spent in solving the constructed SDP and the rest in computing the transition map.

\subsubsection{Algorithm~\ref{alg:ACCPM} with PWA representation}
We import the PWA representation of the closed-loop MPC dynamics which we denote as $x_+ = f_{cl}(x)$ from MPT3 and run Algorithm~\ref{alg:ACCPM} with the mixed-integer formulation of $f_{cl}(x)$ as shown in Section~\ref{sec:PWA_representation}. Since the domain $\mathcal{R}$ is positively invariant, we set $\mathcal{X} = \mathcal{R}$. For Lyapunov function candidates of order $k = 0, 1$, Algorithm~\ref{alg:ACCPM} certifies the non-existence of Lyapunov functions after $1.709$ seconds with $4$ iterations and $28.333$ seconds with $7$ iterations, respectively. With $V^{(k)}(x;P)$ of order $k= 2$, Algorithm~\ref{alg:ACCPM} terminates in $14$ iterations with a valid Lyapunov function candidate and the total running time is $4379.716$ seconds. We plot the counterexamples found in each iteration in Fig.~\ref{fig:MPC_2d_counterexamples}. The accumulated running time of Algorithm~\ref{alg:ACCPM} in each iteration is shown in Fig.~\ref{fig:time_compare}. 


\begin{figure}[htb!] 
	\centering
	\begin{subfigure}[t]{0.49 \columnwidth}
		\centering
		\includegraphics[width = \columnwidth]{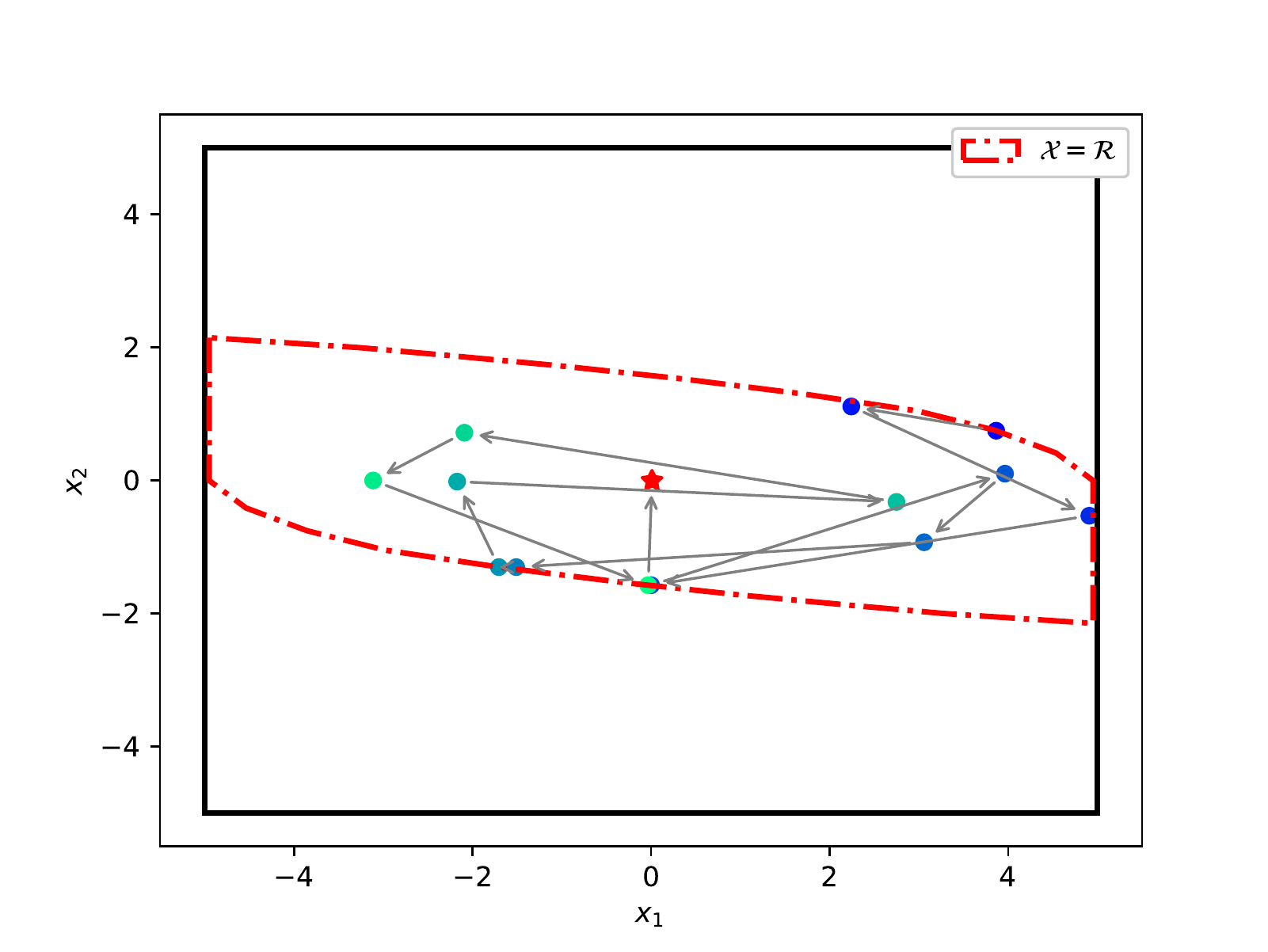}
		\caption{Sequence of counterexamples found in Algorithm~\ref{alg:ACCPM}.}
		\label{fig:MPC_2d_counterexamples}
	\end{subfigure} \hfil
	\begin{subfigure}[t]{0.49 \columnwidth}
		\centering
		\includegraphics[width =  \columnwidth]{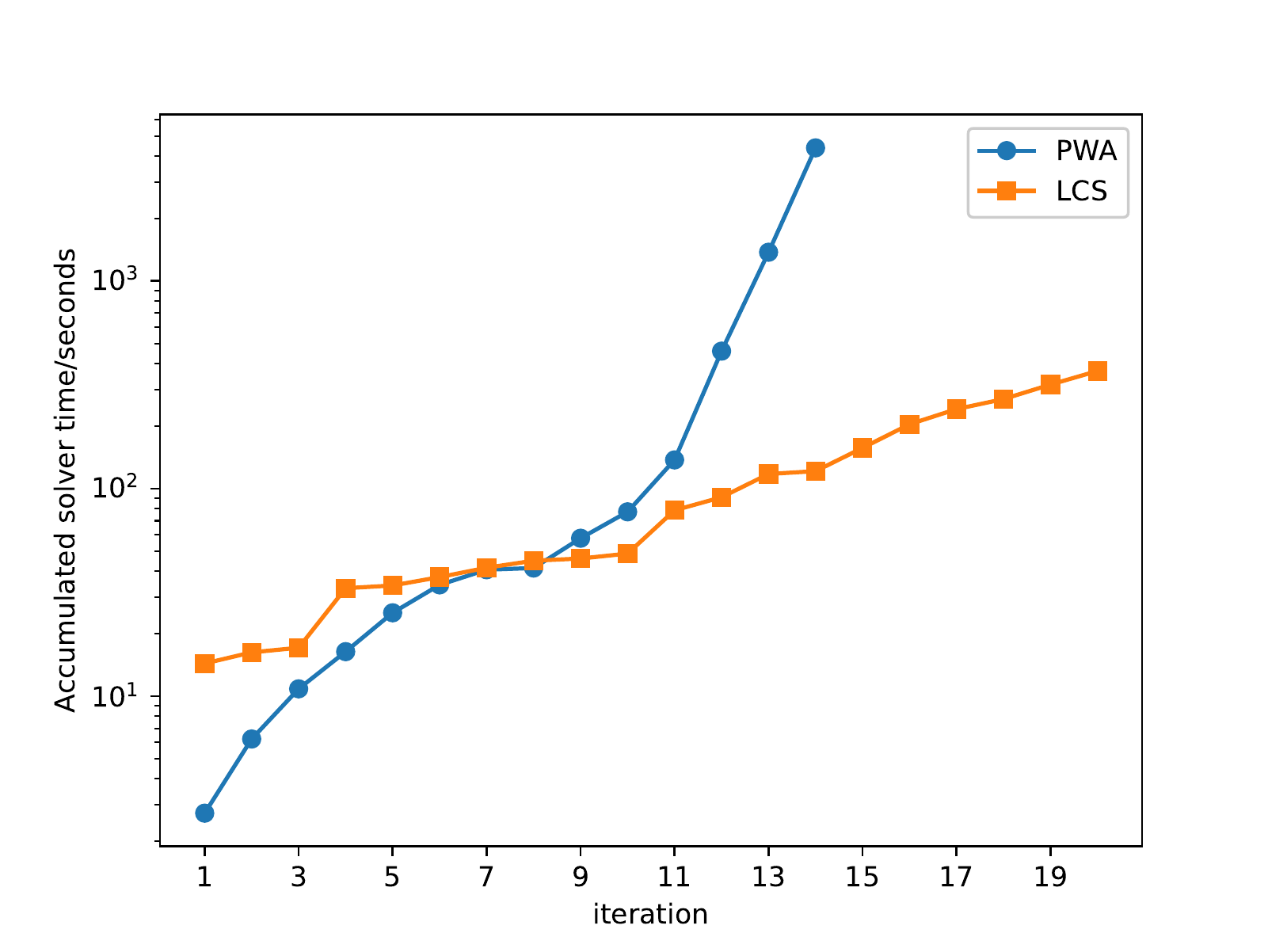}
		\caption{Accumulated running time of Algorithm~\ref{alg:ACCPM}. }
		\label{fig:time_compare}
	\end{subfigure}
	\caption{(a) Sequence of counterexample states found by the verifier in Algorithm~\ref{alg:ACCPM} with the PWA representation of the closed-loop MPC system. (b) Accumulated running time of Algorithm~\ref{alg:ACCPM} in each iteration with mixed-integer formulations induced by the PWA (blue) and LC (orange) representations of the closed-loop MPC system. }
\end{figure}


\subsubsection{Algorithm~\ref{alg:ACCPM} with LCS representation}
As shown in Section~\ref{sec:LCS_representation} and~\cite{simon2016stability}, we can obtain a mixed-integer formulation of the closed-loop MPC dynamics through its LC system representation instead of the PWA one. Based on this mixed-integer formulation, we run Algorithm~\ref{alg:ACCPM} with Lyapunov function candidates $V^{(k)}(x;P)$ of order $k = 2$. The algorithm terminates in $368.618$ seconds with $20$ iterations and returns a valid Lyapunov function which certifies that the domain $\mathcal{R}$ is the ROA. The accumulated running time at each iteration is summarized in Fig.~\ref{fig:time_compare}. 


\subsubsection{Discussion}
Fig.~\ref{fig:time_compare} shows that Algorithm~\ref{alg:ACCPM} depends on the mixed-integer representation of the system. In this example, we need $211$ binary variables to describe the map $x_+ = f_{cl}(x)$ using the PWA representation, while with the LCS representation, we only need to use $64$ binary variables to describe the map $u = \pi_{MPC}(x)$, and hence the closed-loop dynamics. However, when compared with the LMI-based method which synthesizes a PWQ discontinuous  Lyapunov functions in $38.890$ seconds, Algorithm~\ref{alg:ACCPM} is rather inefficient. This is partly due to that the continuous Lyapunov function candidate class $V^{(k)}(x;P)$ is more conservative than the discontinuous one~\cite{biswas2005survey}. To compensate for the conservatism, we need to apply high order $V^{(k)}(x;P)$ which increases the complexity of the MIQP. In the next subsection, we show that Algorithm~\ref{alg:ACCPM} can synthesize a Lyapunov function when the LMI-based method fails.

\subsection{A 4-dimensional closed-loop MPC system}
Consider model predictive control of a randomly generated $4$-dimensional open-loop unstable linear system given by
\begin{equation*}
A = \begin{bmatrix}
0.4346 &   -0.2313&   -0.6404&    0.3405 \\
-0.6731&    0.1045&  -0.0613 &   0.3400 \\
-0.0568&    0.7065&  -0.0861&    0.0159\\
0.3511 &   0.1404 &   0.2980&   1.0416
\end{bmatrix}, \ 
B = \begin{bmatrix}
0 \\ -0.0065 \\ -0.5238 \\ 0.4605
\end{bmatrix}.
\end{equation*}
With the state constraint $\lVert x \rVert_\infty \leq 5$ and the input constraint $-1 \leq u \leq 1$, we design the MPC controller by choosing horizon $T = 10$, stage cost with $Q = 10I, R = 1$, terminal cost with $P_\infty$ and the terminal set as in Section~\ref{sec:2d_example}. The explicit MPC controller is constructed through MPT3 and has $193$ partitions in its domain $\mathcal{R}$, which is validated to be positive invariant under the closed-loop dynamics. Although it took only $29.9$ seconds to compute the transition map, the constructed SDP is ill-posed and Mosek failed to find a feasible solution. 

Then we apply Algorithm~\ref{alg:ACCPM} to synthesize a Lyapunov function $V^{(k)}(x;P)$ with $k = 1$ and the LCS representation of the closed-loop MPC system. The ROI is chosen as $\mathcal{X} = \mathcal{R}$. Algorithm~\ref{alg:ACCPM} terminates in $9$ iterations with a valid Lyapunov function candidate and therefore proves that the closed-loop system is asymptotically stable in the domain $\mathcal{R}$. The total running time is $680.515$ seconds and the accumulated running time in each iteration is plotted in Fig.~\ref{fig:MPC_4d_time}.

\begin{figure}
\centering
\includegraphics[width = 0.5 \textwidth]{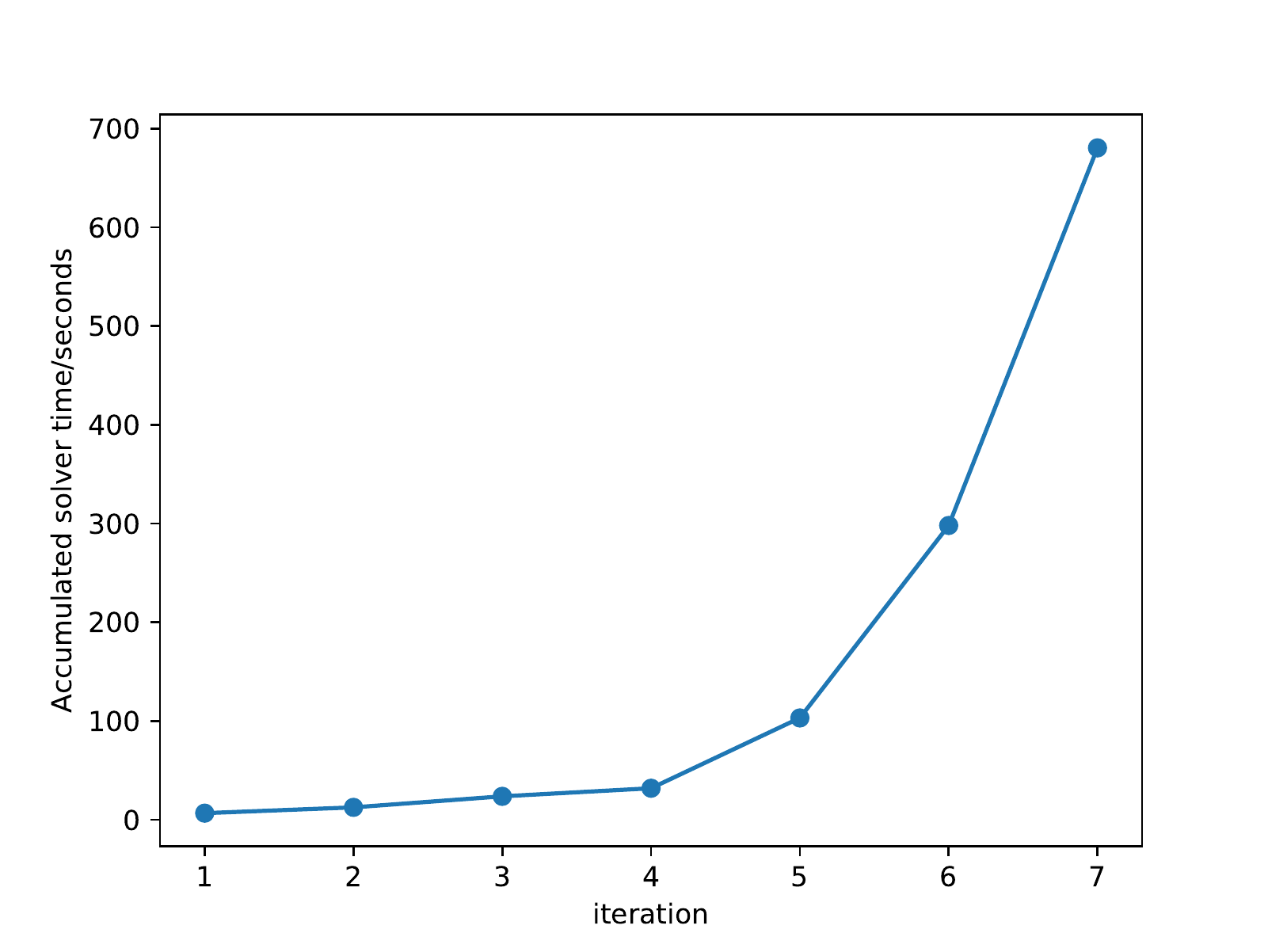}
\caption{Accumulated solver time of Algorithm~\ref{alg:ACCPM} in each iteration for the $4$-dimensional closed-loop MPC system.}
\label{fig:MPC_4d_time}
\end{figure}

\subsection{Neural network controlled PWA system}
We use a hybrid system example from~\cite{marcucci2017approximate} and consider the inverted pendulum shown in Fig.~\ref{fig:pendulum_and_controller} with parameters $m = 1, \ell = 1, g = 10, k = 100, d = 0.1$. We denote the angle and the angular velocity of the pendulum by $q$ and $\dot{q}$, respectively, and define the system state as $x = (q, \dot{q})$. By linearizing the dynamics of the inverted pendulum around $x = 0$, we obtain a hybrid system which has two modes: not in contact with the elastic wall (mode $1$) and in contact with the elastic wall (mode $2$). After discretizing the model using the explicit Euler scheme with a sampling time $h = 0.01$, a PWA model $x_+ = \psi(x, u)$ of the form~\eqref{eq:pwa_dynamics} is obtained with the following parameters
\begin{equation} \label{eq:pendulum_dynamics}
\begin{aligned}
& A_1 = \begin{bmatrix}
1 & 0.01 \\ 0.1 & 1
\end{bmatrix}, B_1 = \begin{bmatrix}
0 \\ 0.01 
\end{bmatrix}, c_1 = \begin{bmatrix}
0 \\ 0
\end{bmatrix}, \\
& \mathcal{R}_1 = \{x \vert [-0.2 \ -1.5]^\top \leq x \leq [0.1 \ 1.5]^\top \}.
\\ 
& A_2 = \begin{bmatrix}
1 & 0.01\\
-0.9 & 1 
\end{bmatrix}, B_2 = \begin{bmatrix}
0 \\ 0.01
\end{bmatrix}, c_2 = \begin{bmatrix}
0 \\ 0.1
\end{bmatrix}, \\
& \mathcal{R}_2 = \{x \vert [0.1 \ -1.5]^\top \leq x \leq [0.2 \ 1.5]^\top\}.
\end{aligned}
\end{equation}

\begin{figure}[htb!] 
	\centering
	\begin{subfigure}[t]{0.49 \columnwidth}
		\centering
		\includegraphics[width = 0.65 \columnwidth]{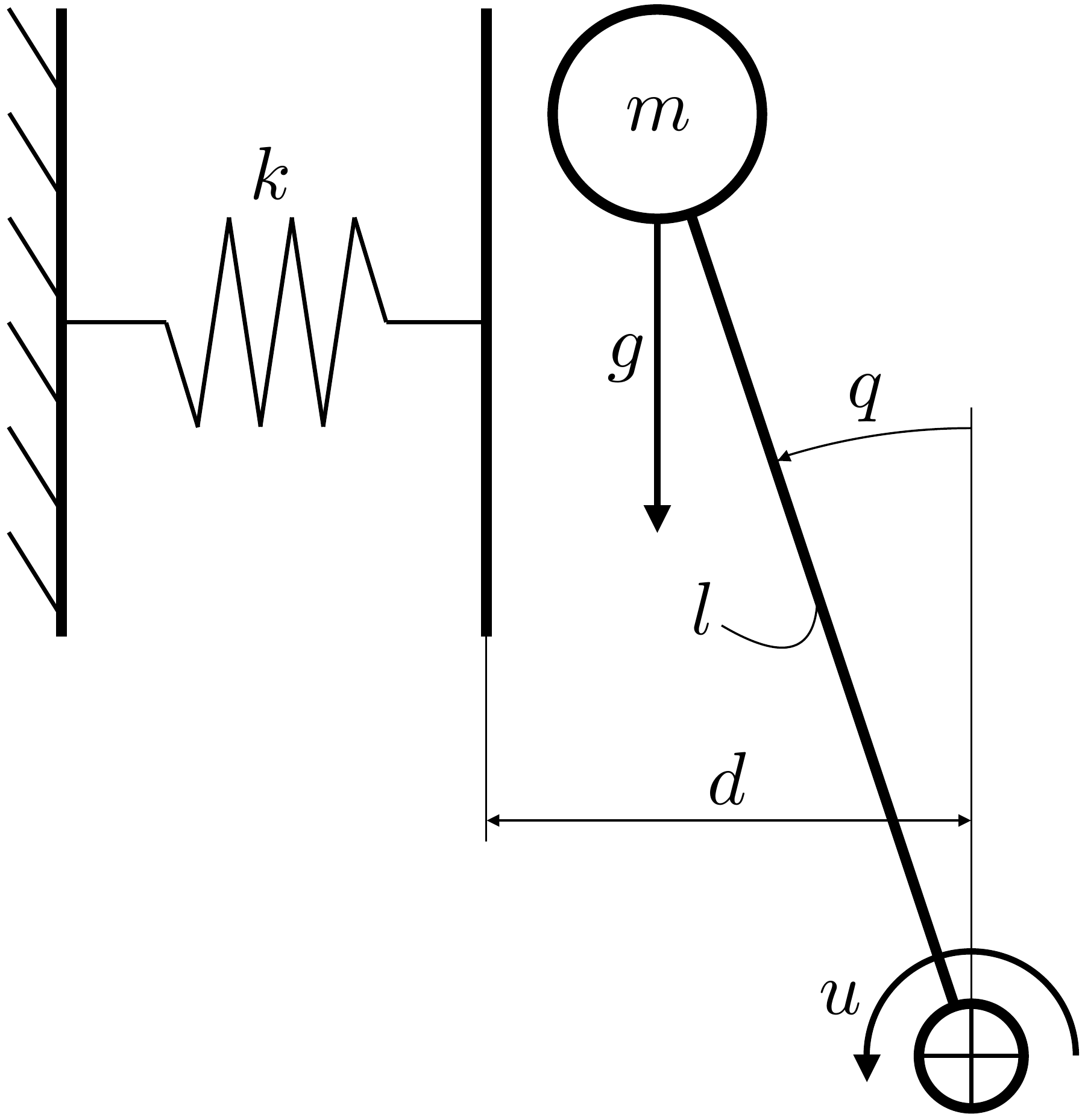}
	\end{subfigure} \hfil
	\begin{subfigure}[t]{0.49 \columnwidth}
		\centering
		\includegraphics[width =  \columnwidth]{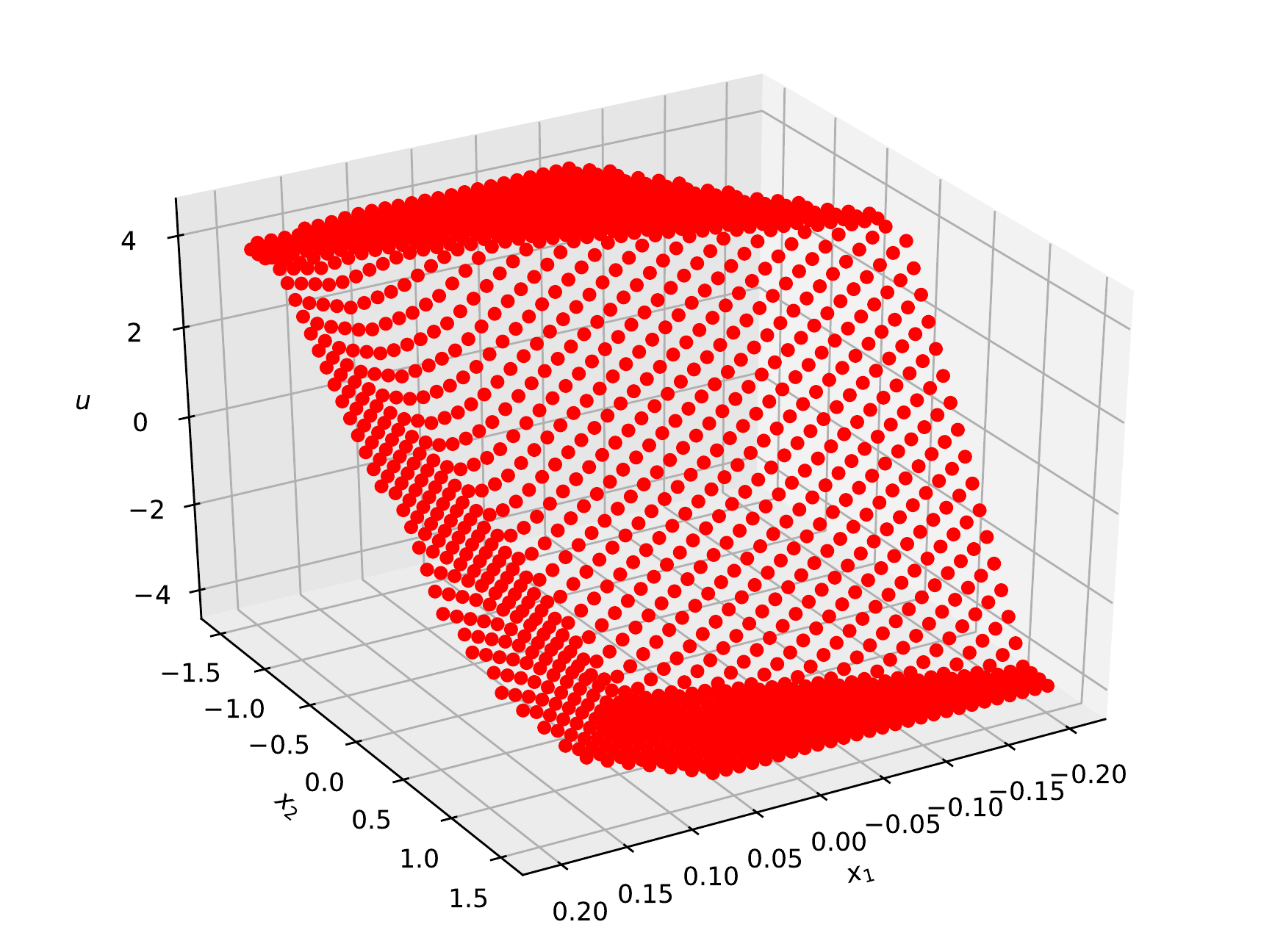}
	\end{subfigure}
	\caption{A ReLU neural network (right) that approximates a hybrid MPC controller is applied on the inverted pendulum system with an elastic wall (left).}
	\label{fig:pendulum_and_controller}
\end{figure}

We then synthesize a hybrid MPC controller $\pi_{MPC}(x)$ for the PWA system~\cite{marcucci2019mixed} where the control input constraints are given by $-4 \leq u \leq 4$ and the horizon of MPC is set as $T = 10$. The stage and terminal costs are given by $Q = I, R = 1$, and $P_\infty = \text{DARE}(A_1, B_1, Q, R)$. We evaluate $\pi_{MPC}(x)$ on a uniform $40 \times 40$ grid samples from the state space and let the reference ROI $\mathcal{X}_0$ be the convex hull of all the feasible state samples. Then the ROI $\mathcal{X} = \gamma \mathcal{X}_0$ with $0 < \gamma \leq 1$ is applied to guide the search for an estimate of ROA.

A total number of $1354$ feasible samples of state and control input pairs $(x, \pi_{MPC}(x))$ are generated to train a ReLU neural network $\pi(x)$ in Keras~\cite{chollet2015keras} to approximate the MPC controller. The neural network has $2$ hidden layers with $20$ neurons in each layer and its output layer bias term is modified after training to guarantee $\pi(0) = 0$. We plot the neural network controller in Fig.~\ref{fig:pendulum_and_controller} and set $\epsilon = 0.0158$ in~\eqref{eq:MIQP_f} since the closed-loop dynamics is linear and asymptotically stable inside $B_\epsilon$. 

For Lyapunov function candidates of order $k=0$ and $k=1$, we run Algorithm~\ref{alg:ACCPM} with ROI $\mathcal{X} = \gamma \mathcal{X}_0$ of varying values of $\gamma$. For the quadratic function class $V^{(0)}(x;P)$, the largest ROI is given by $\mathcal{X} = 0.86 \mathcal{X}_0$ through bisection with which Algorithm~\ref{alg:ACCPM} terminates in $16$ iterations with a total running time of $13.687$ seconds. The corresponding estimate of ROA is shown in Fig.~\ref{fig:IP_quad_ROA}. For the PWQ function class $V^{(1)}(x;P)$, the largest ROI is given by $\mathcal{X} = 1.0 \mathcal{X}_0$ in which case Algorithm~\ref{alg:ACCPM} terminates in $11$ iterations with a total running time of $90.917$ seconds. The estimate of ROA obtained by the found PWQ Lyapunov function candidate is shown in Fig.~\ref{fig:IP__ROA}. It shows that the synthesized Lyapunov function in $V^{(1)}(x;P)$ is less conservative compared with the one in $V^{(0)}(x;P)$ and Algorithm~\ref{alg:ACCPM} can obtain non-trivial estimates of ROA for the neural network controlled hybrid systems.

\begin{figure}
	\centering
	\begin{subfigure}{0.49\columnwidth}
		\centering
		\includegraphics[width = 0.99 \linewidth]{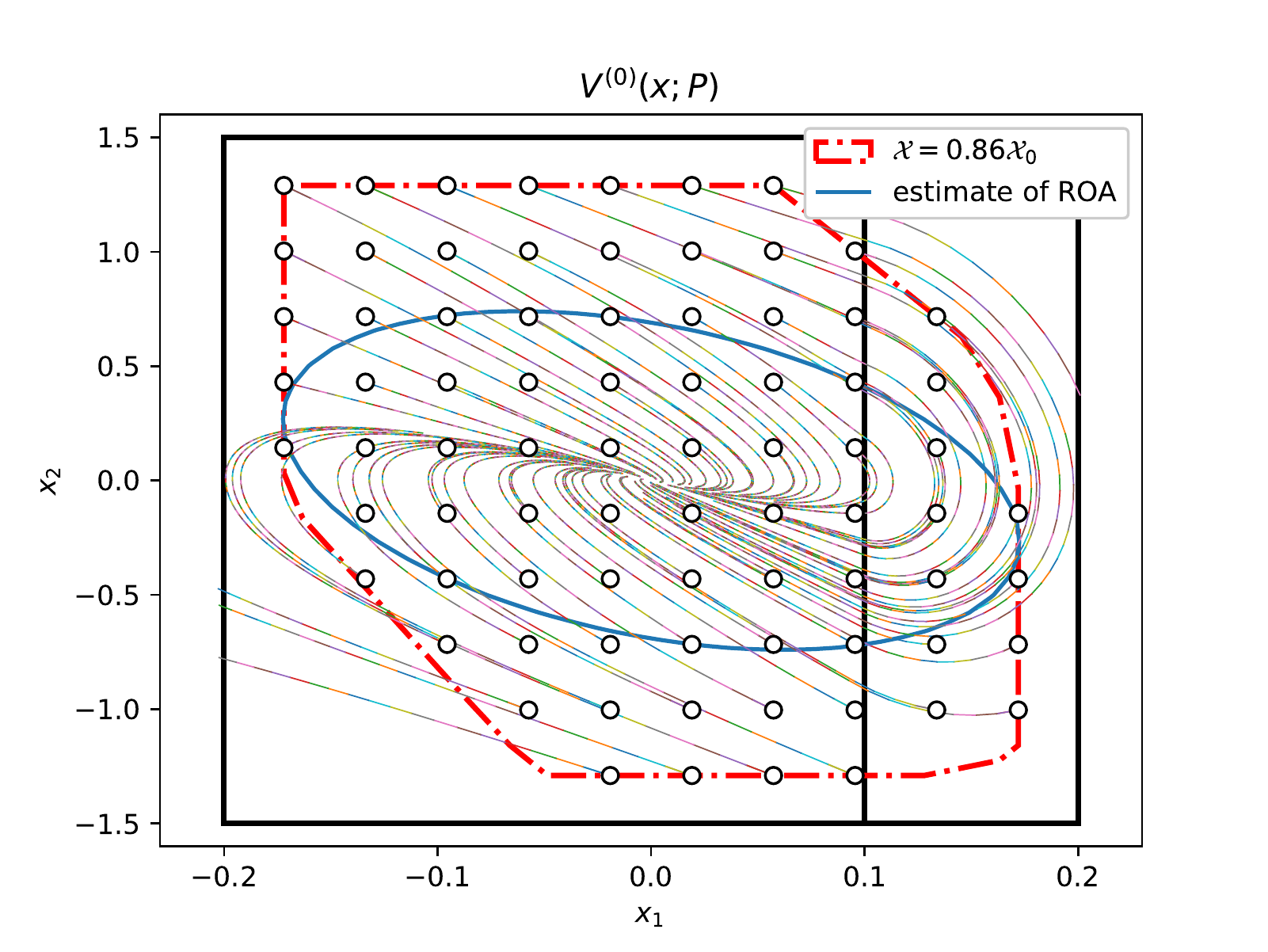}
		\caption{Estimate of ROA of the neural network controlled system by a quadratic Lyapunov function in $V^{(0)}(x;P)$. }
		\label{fig:IP_quad_ROA}
	\end{subfigure} 
	\begin{subfigure}{0.49\columnwidth}
		\centering
		\includegraphics[width = 0.99 \linewidth]{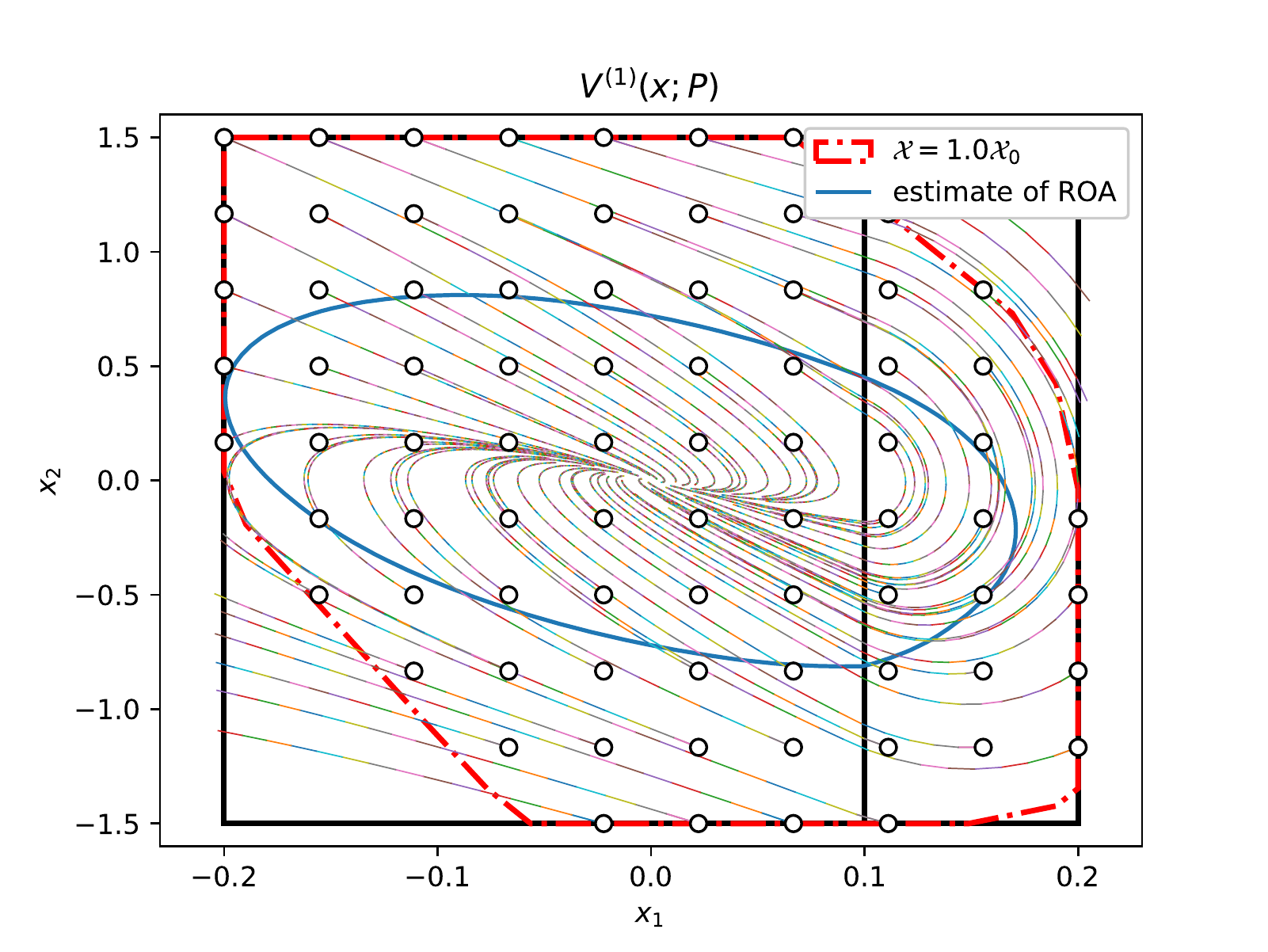}
		\caption{Estimate of ROA of the neural network controlled system by a PWQ Lyapunov function in $V^{(1)}(x;P)$. }
		\label{fig:IP__ROA}
	\end{subfigure}
	\caption{Estimates of ROA found by Algorithm~\ref{alg:ACCPM} with Lyapunov function candidates in $V^{(0)}(x;P)$ and $V^{(1)}(x;P)$. The partitions of the state space are marked by the black boxes. Simulated closed-loop trajectories with the neural network controller are plotted for a grid of initial conditions.}
	\label{fig:IP_ROA}
\end{figure}

\section{Conclusion}
We have proposed a learning-based method to learn Lyapunov functions for autonomous hybrid systems that have a mixed-integer formulation, including piecewise affine, linear complementarity, mixed logical dynamical systems and ReLU neural networks. By designing the method according to the analytic center cutting-plane method, we show that the proposed algorithm is guaranteed to find a Lyapunov function in a finite number of steps when the set of Lyapunov functions is full-dimensional in the parameter space. Our method is an alternative to the LMI-based Lyapunov function synthesis approach which relies on the piecewise affine representation of hybrid systems.


\small
\bibliographystyle{ieeetr}
\bibliography{acmart}

\end{document}